\documentclass[12pt]{article}

\usepackage{amsfonts,mathrsfs}
\usepackage{amssymb,amsmath,amsbsy,amsthm}
\usepackage{dsfont}
\usepackage{blkarray}
\usepackage{tikz}
\usepackage{tkz-euclide}
\usepackage{tikz-cd}
\usetikzlibrary{patterns}
\usepackage{ae}
\usepackage{bm}
\usepackage{graphicx}
\usepackage{float}
\usepackage{cite}
 \usepackage{indentfirst}
\usepackage{lineno}
\usepackage{titlesec}
\usepackage{enumitem}
\usepackage{color}
\usepackage{aliascnt}
\usepackage{varioref}
\usepackage{hyperref}
\hypersetup{colorlinks=true,linkcolor=cyan,anchorcolor=cyan,citecolor=red,CJKbookmarks=True,
,urlcolor=cyan}
\usepackage{cleveref}

\numberwithin{equation}{section}
\def\rank{\mbox{\rm rank}}

\def\gcd{\mbox{\rm gcd}}

\def\int{\mbox{\rm int}}

\def\max{\mbox{\rm max}}
\def\diag{\mbox{\rm diag}}
\def\det{\mbox{\rm det}}
\def\And{\mbox{\rm ~and~}}

\def\I{\mbox{\rm (\hspace{0.2mm}I\hspace{0.2mm})}\,}

\def\Col{\mbox{\rm Col\,}}

\def\({\mbox{\rm (}}\def\){\mbox{\rm )}}

\makeatletter

\newcommand{\Rmnum}[1]{\expandafter\@slowromancap\romannumeral #1@}
\makeatother
\newtheorem{theorem}{Theorem}[section]
\newaliascnt{lemma}{theorem}
\newtheorem{lemma}[lemma]{Lemma}
\aliascntresetthe{lemma}

\newaliascnt{proposition}{theorem}
\newtheorem{proposition}[proposition]{Proposition}
\aliascntresetthe{proposition}

\newaliascnt{fact}{theorem}

\aliascntresetthe{fact}

\newaliascnt{definition}{theorem}

\aliascntresetthe{definition}

\newaliascnt{conjecture}{theorem}

\aliascntresetthe{conjecture}

\newaliascnt{corollary}{theorem}
\newtheorem{corollary}[corollary]{Corollary}
\aliascntresetthe{corollary}

\newaliascnt{claim}{theorem}

\aliascntresetthe{claim}

\newaliascnt{problem}{theorem}

\aliascntresetthe{problem}

\newaliascnt{question}{theorem}

\aliascntresetthe{question}

\newaliascnt{remark}{theorem}

\aliascntresetthe{remark}

\newaliascnt{example}{theorem}

\aliascntresetthe{example}

\newaliascnt{notation}{theorem}

\aliascntresetthe{notation}

\begin{document}

\begin{center}
{\Large\bf
Characteristic quasi-polynomials of truncated arrangements}\\[7pt]
\end{center}
\vskip3mm

\begin{center}
Ying Cao$^{1}$\quad and\quad Houshan Fu$^{2,*}$\\[8pt]
 $^{1,2}$School of Mathematics and Information Science, Guangzhou University\\
Guangzhou 510006, Guangdong, P. R. China\\[12pt]

$^{1}$Department of Mathematics, Hong Kong University of Science and Technology\\
Clear Water Bay, Hong Kong, P. R. China\\[15pt]
 
$^*$Correspondence to be sent to: fuhoushan@gzhu.edu.cn \\
E-mail: $^1$ycaobf@connect.ust.hk\\[15pt]
\end{center}
\vskip 3mm
\begin{abstract}
Given an (affine) integral arrangement $\mathcal{A}$ in $\mathbb{R}^n$, the reduction of $\mathcal{A}$ modulo an arbitrary positive integer $q$ naturally yields an arrangement $\mathcal{A}_q$ in $\mathbb{Z}_q^n$. Our primary objective is to study the combinatorial aspects of the restriction $\mathcal{A}^{(B,\bm b)}$ to the solution space of $B\bm x=\bm b$, and its reduction $\mathcal{A}_q^{(B,\bm b)}$ modulo $q$. This work generalizes the earlier results of Kamiya, Takemura and Terao, as well as Chen and Wang. 

The purpose of this paper is threefold as follows. Firstly, we derive an explicit counting formula for the cardinality of the complement $M\big(\mathcal{A}_q^{(B,\bm b)}\big)$ of $\mathcal{A}_q^{(B,\bm b)}$; and prove that for all positive integers $q>q_0$, this cardinality coincides with a quasi-polynomial $\chi^{\text{quasi}}\big(\mathcal{A}^{(B,\bm b)},q\big)$ in $q$ with a period $\rho_C$. Secondly, we weaken Chen and Wang's original hypothesis $a \mid b$ to a strictly more general condition $\gcd(a,\rho_C)\mid \gcd(b,\rho_C)$,  and introduce the concept of  combinatorial equivalence for positive integers. Within this framework, we establish three unified comparison relations: between the unsigned coefficients of $\chi^{\text{quasi}}\big(\mathcal{A}^{(B,\bm b)},a\big)$ and $\chi^{\text{quasi}}\big(\mathcal{A}^{(B,\bm b)},b\big)$; between the unsigned coefficients of distinct constituents of $\chi^{\text{quasi}}\big(\mathcal{A}^{(B,\bm b)},q\big)$; and between the cardinalities of $M\big(\mathcal{A}_q^{(B,\bm b)}\big)$ and $M\big(\mathcal{A}_{pq}^{(B,\bm b)}\big)$. Thirdly, using our method, we revisit the enumerative aspects of group colorings and nowhere-zero nonhomogeneous form flows from the early work of  Forge, Zaslavsky and Kochol.
\vspace{1ex} \vspace{1ex}
\end{abstract}
\section{Introduction}\label{Sec1}
A {\bf subspace arrangement} is a finite collection $\mathcal{A}=\{H_1,H_2,\ldots,H_m\}$  of affine subspaces of a vector space $V$ over  a field $\mathbb{F}$. If each $H_i$ is a hyperplane (codimension-one subspace) of $V$, $\mathcal{A}$ is called a {\bf hyperplane arrangement}. The {\bf complement} $M(\mathcal{A})$ of $\mathcal{A}$ is defined as
\[
M(\mathcal{A}):=V-\bigcup_{H\in\mathcal{A}}H.
\]
A concept very closely related to $M(\mathcal{A})$ is the {\bf characteristic polynomial} of $\mathcal{A}$, defined by 
\[
\chi(\mathcal{A},t):=\sum_{\mathcal{B}\subseteq\mathcal{A},\,\bigcap_{H\in\mathcal{B}}H\ne\emptyset}(-1)^{|\mathcal{B}|}t^{\dim(\bigcap_{H\in\mathcal{B}}H)}.
\]
Roughly speaking, $\chi(\mathcal{A},t)$ measures the ``size" of $M(\mathcal{A})$ (see \cite{ER1998}). Details regarding subspace arrangements and characteristic polynomials can be found in \cite{Athanasiadis1996,BE1997,Stanley2007}.

Our central object of the study is integral arrangements. Let $A\in\mathcal{M}_{m\times n}(\mathbb{Z})$ be an $m\times n$ matrix with integral entries, and let $\bm a\in\mathbb{Z}^m$ be an integral vector. The matrix $[A,\bm a]$ defines a hyperplane arrangement $\mathcal{A}:=\mathcal{A}(A,\bm a)$ consisting of $m$ hyperplanes $H_i$ in $\mathbb{R}^n$, referred to as an {\bf integral arrangement}.  Specifically, each hyperplane $H_i$ in $\mathcal{A}$ is given by
\[
H_i:a_{i1}x_1+a_{i2}x_2+\cdots+a_{in}x_n=a_i;\quad a_{ij},a_i\in\mathbb{Z},\;1\le i\le m,\;1\le j\le n.
\]
We call the augmented matrix $[A,\bm a]$ the {\bf defining matrix} of $\mathcal{A}$.

Let $q\in\mathbb{Z}_{>0}$ and denote by $\mathbb{Z}_q:=\mathbb{Z}/q\mathbb{Z}$ the additive abelian group of integers modulo $q$. For any $a\in\mathbb{Z}$, the {\bf $q$-reduction} of $a$ is denoted by $[a]_q:=a+q\mathbb{Z}\in\mathbb{Z}_q$. In the same way, for an integral matrix or vector $A'$, denote by $[A']_q$ the entry-wise $q$-reduction of $A'$. Then the $q$-reductions of $A$ and $\bm a$ naturally yield an arrangement 
\[
\mathcal{A}_q=\mathcal{A}_q(A,\bm a):=\{H_{i,q}\mid i=1,2\ldots,m\},
\]
of $m$ ``hyperplanes" $H_{i,q}$ in $\mathbb{Z}_q^n$, where
\[
H_{i,q}:=\big\{\bm x\in\mathbb{Z}_q^n\mid [\bm\alpha^T_i]_q\bm x=[a_i]_q\big\} \quad\text{with}\quad\bm\alpha^T_i=(a_{i1},\ldots,a_{in}).
\]
While  $H_{i,q}$ is not necessarily a hyperplane, we follow a common terminological convention by referring to $H_{i,q}$ as a hyperplane. 

Originated from Crapo and Rota's work \cite{CR1970}, Athanasiadis in \cite{Athanasiadis1996} systematically developed a finite field method to compute the characteristic polynomial of an integral arrangement. This method reduces the computation of the characteristic polynomial to a simple counting problem in a vector space over a finite field, and was also independently discovered by  Bj\"{o}rner and Ekedahl \cite{BE1997}. It is natural to consider such problems on the more general abelian group $\mathbb{Z}_q$, where $q$ is not necessarily power of a prime. Later, Athanasiadis generalized the previous work in \cite[Theorem 2.1]{Athanasiadis1999}: the characteristic polynomial of $\mathcal{A}$ can be computed by enumerating points in the complement of $\mathcal{A}_q$ for large enough integers $q$ that are relatively prime to a constant depending only on $\mathcal{A}$. The extension to truncated arrangements will be presented in \autoref{Char-Quasi}.

Recently, Kamiya, Takemura and Terao \cite{KTT2008,KTT2011} extended Athanasiadis's idea to all positive integers $q$. They showed that there exist non-negative integers $\rho_A$ and $q_0$, both depending only on $[A,\bm a]$, such that the cardinality of the complement of $\mathcal{A}_q$ is a quasi-polynomial $\chi^{\text{quasi}}(\mathcal{A},q)$ in $q$ with a period $\rho_A$ for all integers $q>q_0$.  This quasi-polynomial is named the {\bf characteristic quasi-polynomial} in \cite{KTT2008}, and the period is called the {\bf lcm period} in \cite{KTT2010}. However, the lcm period is not necessarily the minimum period. This relevant topic was well-studied by  Higashitani, Tran and Yoshinaga  \cite{HTY2023}.

In fact, characteristic quasi-polynomials encode significantly more information than characteristic polynomials. For instance, each constituent of the characteristic quasi-polynomial $\chi^{\text{quasi}}(\mathcal{A},q)$ has a different combinatorial interpretation associated to the toric arrangement, as shown in \cite{LTY2021,TY2019}. In particular, the first constituent correspond precisely to the characteristic polynomial $\chi(\mathcal{A},t)$ in \cite{KTT2008,KTT2011}. Chen and Wang  investigated the arithmetic properties of  $\chi^{\text{quasi}}(\mathcal{A},q)$ in \cite{CW2012}, with a particular focus on its coefficients and those of its constituents. To better address this topic, they primarily concentrated on the so-called truncated integral arrangement in the central case.

Let $B\in\mathcal{M}_{l\times n}(\mathbb{Z})$ and $\bm b\in\mathbb{Z}^k$. Denote by $V(B,\bm b)$ the solution space of the linear system $B\bm x=\bm b$, i.e.,
\[
V(B,\bm b):=\big\{\bm x\in\mathbb{R}^n\mid B\bm x=\bm b\big\}.
\]
The restriction of $\mathcal{A}=\mathcal{A}(A,\bm a)$ to $V(B,\bm b)$ yields a subspace arrangement $\mathcal{A}^{(B,\bm b)}$ in $V(B,\bm b)$ given by
\[
\mathcal{A}^{(B,\bm b)}:=\big\{H_1^{(B,\bm b)}, H_2^{(B,\bm b)},\ldots, H_m^{(B,\bm b)}\big\},
\]
where each affine subspace $H_i^{(B,\bm b)}$ is the intersection of $H_i$ and $V(B,\bm b)$.  The subspace arrangement $\mathcal{A}^{(B,\bm b)}$ is called the {\bf truncated arrangement} or {\bf truncation} of $\mathcal{A}$ by $[B,\bm b]$. Let 
\[
V_q(B,\bm b):=\big\{\bm x\in\mathbb{Z}_q^n\mid [B]_q\bm x=[\bm b]_q\big\}.
\]
The truncation $\mathcal{A}^{(B,\bm b)}$ automatically induces an arrangement $\mathcal{A}_q^{(B,\bm b)}$ in $V_q(B,\bm b)$, defined by
\[
\mathcal{A}_q^{(B,\bm b)}:=\Big\{H_{1,q}^{(B,\bm b)}, H_{2,q}^{(B,\bm b)},\ldots,H_{m,q}^{(B,\bm b)}\Big\}, 
\]
where $H_{i,q}^{(B,\bm b)}$ denotes the intersection $H_{i,q}\cap V_q(B,\bm b)$ and coincides with the $q$-reduction of $H_i^{(B,\bm b)}$. We need to be especially careful that $V_q(B,\bm b)$ is not necessarily nonempty. If $l=0$, then $\mathcal{A}^{(B,\bm b)}=\mathcal{A}$ and $\mathcal{A}_q^{(B,\bm b)}=\mathcal{A}_q$. 

Chen and Wang \cite{CW2012} focused on the counting function of the complement of $\mathcal{A}_q^{(B,\bm b)}$ in the {\bf central case} where $\bm a$ and $\bm b$ are the zero vectors. In this setting, they proved in \cite[Theorem 1.1, Theorem 2.3]{CW2012} that the cardinality of the complement $M\big(\mathcal{A}_q^{(B,\bm 0)}\big)$ is a quasi-polynomial $\chi^{\text{quasi}}\big(\mathcal{A}^{(B,\bm 0)},q\big)$ in all positive integers $q$, which satisfies the gcd property and extends the main results in \cite{KTT2008}. In \cite[Theorem 1.2]{CW2012}, they further showed that if positive integers $a$ divides $b$, then the unsigned coefficients of the quasi-polynomials $\chi^{\text{quasi}}\big(\mathcal{A}^{(B,\bm 0)},a\big)$ and $\chi^{\text{quasi}}\big(\mathcal{A}^{(B,\bm 0)},b\big)$ have a unified comparison.

Motivated by the preceding work, our primary objective is to generalize their results to the {\bf non-central case}, where $\bm a$ and $\bm b$ are arbitrary integral vectors. Next, we will briefly introduce our main results. The relevant precise definitions and detailed notations will be explained in later sections.

Using the inclusion-exclusion principle, we first present an explicit expression for the cardinality of the complement of $\mathcal{A}_q^{(B,\bm b)}$, as follows:
\[
\#M\big(\mathcal{A}_q^{(B,\bm b)}\big)=\sum_{J\subseteq[m]}(-1)^{\#J}\tilde{d}_J(q)q^{n-r(J)}.
\]
By the elementary divisors method, we then derive that there exist non-negative integers $\rho_C$ and $q_0$, both depending only on $\mathcal{A}^{(B,\bm b)}$, such that for all integers $q>q_0$, the cardinality of $M\big(\mathcal{A}_q^{(B,\bm b)}\big)$ is a quasi-polynomial in $q$ with $\rho_C$ as a period, and its expression is
\[
\chi^{\text{quasi}}\big(\mathcal{A}^{(B,\bm b)},q\big)=\sum_{J:\, J\subseteq[m],\,\bar{r}(J)=r(J)}(-1)^{\#J}\tilde{d}_J(q)q^{n-r(J)},
\]
where $\tilde{d}_J$ is a periodic function. We call it the {\bf characteristic quasi-polynomial}. Similarly, every constituent $f_a\big(\mathcal{A}^{(B,\bm b)},t\big)$ ($a=1,\ldots,\rho_C$) of $\chi^{\text{quasi}}\big(\mathcal{A}^{(B,\bm b)},q\big)$ is a polynomial of the form
\[
f_a\big(\mathcal{A}^{(B,\bm b)},t\big)=\sum_{J:\, J\subseteq[m],\,\bar{r}(J)=r(J)}(-1)^{\#J}\tilde{d}_J(a)t^{n-r(J)}.
\]
As with \cite{CW2012,KTT2008,KTT2011}, we observe that $f_a\big(\mathcal{A}^{(B,\bm b)},t\big)=f_b\big(\mathcal{A}^{(B,\bm b)},t\big)$ when ${\rm gcd}(a,\rho_C)={\rm gcd}(b,\rho_C)$, and the first constituent agrees with the characteristic polynomial of $\mathcal{A}^{(B,\bm b)}$. 

The other core focus of this paper is the interrelationship of the unsigned coefficients of characteristic quasi-polynomials and those of its constituents. Assume ${\rm rank}(B)=r$ and ${\rm rank}\begin{bmatrix}B\\A\end{bmatrix}=s$. From the expressions for $\chi^{\text{quasi}}\big(\mathcal{A}^{(B,\bm b)},q\big)$ and $f_a\big(\mathcal{A}^{(B,\bm b)},t\big)$, we can write them respectively as
\[
\chi^{\text{quasi}}\big(\mathcal{A}^{(B,\bm b)},q\big)=\sum_{j=r}^s(-1)^{j-r}\beta_j(q)q^{n-j}
\]
and
\[
f_a\big(\mathcal{A}^{(B,\bm b)},t\big)=\sum_{j=r}^s(-1)^{j-r}\gamma_j(a)t^{n-j}.
\]
To derive our desired results, we weaken the original hypothesis $a \mid b$ in \cite{CW2012} to a strictly more general condition $\gcd(a,\rho_C)\mid \gcd(b,\rho_C)$,  and further  introduce the concept of  combinatorial equivalence for positive integers. Within this framework, we establish three uniform comparisons as follows:
\begin{itemize}
\item For any integers $a,b>q_0$, if ${\rm gcd}(a,\rho_C)\mid{\rm gcd}(b,\rho_C)$ and $a$ is combinatorially equivalent to $b$ with respect to $\mathcal{A}^{(B,\bm b)}$, then 
\[
0\le \beta_j(a)\le\beta_j(b),\quad\forall\, r\le j\le s.
\]
\item For any integers $1\le a,b\le\rho_C$, if ${\rm gcd}(a,\rho_C)\mid{\rm gcd}(b,\rho_C)$ and  ${\rm gcd}(a,\rho_C)$ is combinatorially equivalent to ${\rm gcd}(b,\rho_C)$ with respect to $\mathcal{A}^{(B,\bm b)}$, then 
\[
0\le \gamma_j(a)\le\gamma_j(b),\quad\forall\, r\le j\le s.
\]
\item For any $p,q\in\mathbb{Z}_{>0}$, if  ${\rm gcd}(p,d_i){\rm gcd}(q,d_i)\mid d_i$ for every invariant factor $d_i$ of $B$ and the system $B\bm x=\bm b \pmod{pq}$ is solvable, then 
\[
\#M\big(\mathcal{A}_q^{(B,\bm b)}\big)\le \#M\big(\mathcal{A}_{pq}^{(B,\bm b)}\big).
\]
\end{itemize}
Moreover, when $H_1^{(B,\bm b)},\ldots,H_m^{(B,\bm b)}$ are codimension-one subspaces of $V(B,\bm b)$ or $l=0$, the coefficients of $f_a\big(\mathcal{A}^{(B,\bm b)},t\big)$ are nonzero and alternate in sign if $\gcd(a,\rho_C)$ is combinatorially equivalent to $1$ with respect to $\mathcal{A}^{(B,\bm b)}$. The result and the first comparison strengthen Chen and Wang's main results and extend them to the non-central case. As far as we know, the last two comparisons are established for the first time.

A closely related topic is nowhere-zero nonhomogeneous form flows and group colorings, introduced by Jaeger, Linial, Payan, and Tarsi \cite{Jaeger1992}, which generalizes the classical nowhere-zero flows and proper vertex colorings in graph theory. Using our method, we revisit the relevant enumerative aspects in earlier works \cite{FZ2007,FZ2016,FRW2025, Kochol2022, Zaslavsky1995, Zaslavsky2003}.

The remainder of the paper is organized as follows. In \autoref{Sec2}, we mainly recall the definitions and basic properties of quasi-polynomials and elementary divisors, and study the cardinality of the complement $M\big(\mathcal{A}_q^{(B,\bm b)}\big)$ of $\mathcal{A}_q^{(B,\bm b)}$ and its associated properties in the non-central case. \autoref{Sec3} focuses on three comparisons: the unsigned coefficients of $\chi^{\text{quasi}}\big(\mathcal{A}^{(B,\bm b)},a\big)$ and $\chi^{\text{quasi}}\big(\mathcal{A}^{(B,\bm b)},b\big)$; the unsigned coefficients of distinct constituents $f_a\big(\mathcal{A}^{(B,\bm b)}, t\big)$; and the cardinalities of $M\big(\mathcal{A}_q^{(B,\bm b)}\big)$ and $M\big(\mathcal{A}_{pq}^{(B,\bm b)}\big)$. \autoref{Sec4} is devoted to discussing the applications of characteristic quasi-polynomials in graph coloring and flow problems.
\section{Characteristic quasi-polynomials}\label{Sec2}
In this section, we start by recalling the definitions and basic properties of quasi-polynomials and elementary divisors, and then focus on studying the cardinality of the complement of $\mathcal{A}_q^{(B,\bm b)}$ in the non-central case and its associated properties.
\subsection{Preliminaries}\label{Sec2-1}
For the reader's convenience, we first survey the relevant definitions and standard facts concerning quasi-polynomials and elementary divisors, which will be used later.

Let us review the definition of quasi-polynomials. A function $f:\mathbb{Z}\to\mathbb{C}$ is called a {\bf quasi-polynomial} if there exist $\rho\in\mathbb{Z}_{>0}$ and polynomials $f_1(t),f_2(t),\ldots,f_\rho(t)\in\mathbb{Q}[t]$ such that for any $q\in\mathbb{Z}_{>0}$,
\[
f(q)=f_j(q)\quad\text{if }\;q\equiv j\pmod{\rho}.
\]
The integer $\rho$ (which is not unique) is referred to as a {\bf quasi-period} of $f$, and the polynomial $f_j(t)$ is called the {\bf$j$-constituent} of the quasi-polynomial $f$. Equivalently, $f$ is of the form
\[
f(q)=d_n(q)q^n+d_{n-1}(q)q^{n-1}+\cdots+d_0(q),
\]
where each $d_j\colon\mathbb{Z}\to\mathbb{Q}$ is a periodic function with integral period $\tau_j$, that is, $d_j(q)=d_j(q+\tau_j)$ for $\tau_j\in\mathbb{Z}_{>0}$. In this case, the above period $\rho$ can be chosen as $\rho={\rm lcm}(\tau_0,\tau_1,\ldots,\tau_n)$. The smallest such $\rho$ is called the {\bf minimum period} of the quasi-polynomial $f$. Moreover, a quasi-polynomial $f$ with a period $\rho$ is said to have the {\bf gcd property} with respect to $\rho$ if every $j$-constituent $f_j(t)\;(1\leq j\leq\rho)$ depends on $j$ only through ${\rm gcd}(j,\rho)$, that is, 
\[
f_a(t)=f_b(t)\quad\text{if}\quad{\rm gcd}(a,\rho)={\rm gcd}(b,\rho).
\]
Further details on quasi-polynomials can be found in \cite{Stanley2012}.

Next, we collect some foundational aspects of elementary divisor theory. Let $M\in\mathcal{M}_{m\times n}(\mathbb{Z})$ be an integral matrix of rank $r$. From the elementary divisor theorem, $M$ can be transformed to its Smith normal form via elementary row and column operations. Specifically, the {\bf Smith normal form} of $M$ is the unique matrix $D=D(M)$ whose first $r$ diagonal entries are non-zero integers $d_1(M),d_2(M),\ldots,d_r(M)$ and remaining entries are zero, such that there exist unimodular matrices $P\in GL_m(\mathbb{Z})$ and $Q\in GL_n(\mathbb{Z})$ satisfy 
\[
PMQ=D\quad \text{and} \quad 0<d_1(M)\mid d_2(M)\mid \cdots\mid d_{r-1}(M)\mid d_r(M).
\]
Each $d_j(M)$ is called the {\bf$j$-th invariant factor} of $M$. In particular, $d(M):=d_r(M)$ is termed the {\bf maximal invariant factor}. Moreover, every $j$-th invariant factor can be computed by the following formula:
\[
\prod_{i=1}^jd_i(M)={\rm gcd}\big\{\det(M')\mid M'\text{ is a }j\times j\text{ submatrix of }M\big\}.
\]
For more information on the topic, see \cite{Lang2002}. 

Recently, Kamiya, Takemura and Terao investigated the relevant theory for the ring $\mathbb{Z}_q$ ($q\in\mathbb{Z}_{>0}$) in  \cite{KTT2011}. They showed in \cite[Proposition 2.1]{KTT2011} that an arbitrary matrix $M\in\mathcal{M}_{m\times n}(\mathbb{Z}_q)$ is equivalent to a diagonal matrix
\[
\diag\big([d_1]_q,\ldots,[d_s]_q,0,\ldots,0\big),
\]
for some integers $0<d_1\mid d_2\mid\cdots\mid d_s\mid q$ with $d_s<q$, where each $d_i$ is uniquely determined by $M$. More precisely, let $M'\in\mathcal{M}_{m\times n}(\mathbb{Z})$ satisfy $[M']_q=M$. Then
\begin{align}\label{eq:invariant factor}
s={\rm max}\big\{i:q\text{ does not divide } d_i(M')\big\}\quad\And\quad d_i={\rm gcd}\big(q,d_i(M')\big).
\end{align}
Consequently, $[d_i]_q\doteq[d_i(M')]_q$, where $\doteq$ stands for equality up to a unit multiplication in $\mathbb{Z}_q$. Likewise, we call $[d_j]_q$ the {\bf $j$-th invariant factor} of $M$.

Let $M\in\mathcal{M}_{m\times n}(\mathbb{Z})$ and $\bm c\in\mathbb{Z}^m$.  Define $\Col_q(M)$ as the column space generated by the columns of $[M]_q$ over $\mathbb{Z}_q$. Obviously, $\Col_q(M)$ is a submodule of the $\mathbb{Z}_q$-module $\mathbb{Z}_q^m$. Kamiya, Takemura and Terao  \cite{KTT2011} provided an alternative characterization of the solvability of the system $M\bm x=\bm c\pmod q$ associated with the invariant factors of $[M,\bm c]_q$ and $[M,\bm 0]_q$.
\begin{lemma}[\cite{KTT2011}, Lemma 2.3]\label{Solvable}
Let $q\in\mathbb{Z}_{>0}$, $M\in\mathcal{M}_{m\times n}(\mathbb{Z})$ and $\bm c\in\mathbb{Z}^m$. Then $[\bm c]_q\in\Col_q(M)$ if and only if the invariant factors of $[M,\bm c]_q$ are the same as those of $[M,\bm 0]_q$.
\end{lemma}
To obtain our results, the following theorem of Tompson \cite{Thompson1979} is also required, known as the Interlacing Divisibility Theorem.
\begin{theorem}[\cite{Thompson1979}, Interlacing Divisibility Theorem]\label{IDT}
Let $M\in\mathcal{M}_{m\times n}(\mathbb{Z})$ be an integral matrix of rank $r$ and $N=\begin{bmatrix} M\\\bm c^T\end{bmatrix}$ with $\bm c\in\mathbb{Z}^n$. Then
\[
d_1(N)\mid d_1(M)\mid d_2(N)\mid d_2(M)\mid\cdots\mid d_r(N)\mid d_r(M)\mid d_{r+1}(N),
\]
where $d_{r+1}(N)=0$ if ${\rm rank}(N)=r$. In particular, the maximal invariant factors $d(M)$ and $d(N)$ satisfy the following relation:
\begin{align*}
d(M)\mid d(N)\quad&\text{ if }\quad {\rm rank}(N)={\rm rank}(M)+1,\\
d(N)\mid d(M)\quad&\text{ if }\quad{\rm rank}(N)={\rm rank}(M).
\end{align*}
\end{theorem}

\subsection{Characteristic quasi-polynomials of truncated arrangements}\label{Sec2-2}
In this subsection, we first derive a counting formula for $\#\big(V_q(B,\bm b)-\bigcup_{j=1}^mH_{j,q}^{(B,\bm b)}\big)$ via the inclusion-exclusion principle. By the theory of elementary divisors,  we then show that the counting function can be uniformly represented as a quasi-polynomial for all integers $q>q_0$ in \autoref{Main1}. Finally, we present explicit expressions for the constituents of this polynomial, and further explore their interrelationships in \autoref{Main2}.

Let us first analyze the number of solutions in $\mathbb{Z}_q^n$ to the system $[M]_q\bm x=[\bm c]_q$.
\begin{lemma}\label{Compute}
Let $q\in\mathbb{Z}_{>0}$, $\bm c\in\mathbb{Z}^m$, and $M\in\mathcal{M}_{m\times n}(\mathbb{Z})$ be an integral matrix of rank $r$. Then the number of solutions of the system $[M]_q\bm x=[\bm c]_q$ in $\mathbb{Z}_q^n$ is
\[
\#V_q(M,\bm c)=\begin{cases}q^{n-r}\prod_{j=1}^{r}{\rm gcd}\big(q,d_j(M)\big),&\text{ if } [\bm c]_q\in\Col_q(M);\\
0,&\text{ otherwise}.
\end{cases}
\]
\end{lemma}

\begin{proof}
If $[\bm c]_q\notin\Col_q(M)$, then the system $[M]_q\bm x=[\bm c]_q$ is unsolvable in $\mathbb{Z}_q^n$. Hence, $\#V_q(M,\bm c)=0$ in this case. If $[\bm c]_q\in\Col_q(M)$, then the system $[M]_q\bm x=[\bm c]_q$ has solutions in $\mathbb{Z}_q^n$. Thus, for some $\bm x_0\in\mathbb{Z}_q^m$ satisfying $[M]_q\bm x_0=[\bm c]_q$, we have 
\[
V_q(M,\bm c)=\bm x_0+\big\{\bm x\in\mathbb{Z}^m_q:[M]_q\bm x=[\bm 0]_q\big\}=\bm x_0+V_q(M,\bm 0).
\]
This implies $\#V_q(M,\bm c)=\#V_q(M,\bm 0)$. On the other hand, from the theory of elementary divisors, there exist unimodular matrices $P\in GL_m(\mathbb{Z})$ and $Q\in GL_n(\mathbb{Z})$ such that
\[
PMQ=D={\rm diag}\big(d_1(M),\ldots,d_r(M),0,\ldots,0\big).
\]
It follows that $V_q(M,\bm 0)=V_q(PM,\bm 0)$. Since $Q$ is an invertible matrix, the group homomorphism
\[
V_q(M,\bm 0)=V_q(PM,\bm 0)\to V_q(PMQ,\bm 0),\quad\bm x\mapsto Q^{-1}\bm x
\]
is an isomorphism. This indicates $\#V_q(M,\bm 0)=\#V_q(D,\bm 0)$. The system $D\bm x=\bm 0$ in $\mathbb{Z}_q^n$ consists of the equations $d_j(M)x_j=0$ with $j=1,2,\ldots,r$. By elementary number theory, every equation $d_j(M)x_j=0$ has exactly ${\rm gcd}\big(q,d_j(M)\big)$ solutions in $\mathbb{Z}_q$. So, we have 
\[
\#V_q(M,\bm c)=\#V_q(D,\bm 0)=q^{n-r}\prod_{j=1}^r{\rm gcd}\big(q,d_j(M)\big).
\]
We finish the proof.
\end{proof}

To state our results more conveniently, we introduce some necessary notations. From now on, unless otherwise stated, we always assume $A\in\mathcal{M}_{m\times n}(\mathbb{Z})$, $B\in\mathcal{M}_{l\times n}(\mathbb{Z})$, $\bm a\in\mathbb{Z}^m$ and $\bm b\in\mathbb{Z}^l$. For any $J\subseteq [m]$, define 
\[
C_J:=\begin{bmatrix}B\\ A_J\end{bmatrix}\in\mathcal{M}_{(l+\#J)\times n}(\mathbb{Z}),\;\bm c_J:=\begin{bmatrix}\bm b\\\bm a_J\end{bmatrix}\in\mathbb{Z}^{l+\#J}\And \bar{C}_J:=[C_J\;\bm c_J]\in\mathcal{M}_{(l+\#J)\times {(n+1)}}(\mathbb{Z}), 
\]
where $A_J\in\mathcal{M}_{\#J\times n}(\mathbb{Z})$ is the submatrix of $A$ consisting of the rows indexed by $J$ and $\bm a_J\in\mathbb{Z}^{\#J}$ is the subvector of $\bm a$ consisting of the entries indexed by $J$. For simplicity, we write $C:=C_{[m]}$ and $\bar{C}:=\bar{C}_{[m]}$. Set 
\[
r(J):={\rm rank}(C_J)\quad\And\quad\bar{r}(J):={\rank}(\bar{C}_J).
\]
Let $0<d_{J,1}\mid d_{J,2}\mid\cdots\mid d_{J,r(J)}$ be the invariant factors of $C_J$, and let $0<\bar{d}_{J,1}\mid \bar{d}_{J,2}\mid\cdots\mid {\bar d}_{J,\bar{r}(J)}$  be the invariant factors of $\bar{C}_J$. Further set
\[
\rho_C:={\rm lcm}\big(d_{J,r(J)}:J\subseteq[m],\;\bar{r}(J)=r(J)\big)
\]
and
\[
q_0=q_0(\bar{C}):={\rm max}\big\{\bar{d}_{J,\bar{r}(J)}:J\subseteq[m],\;\bar{r}(J)=r(J)+1\big\}.
\]
When $\big\{J\subseteq[m]:\bar{r}(J)=r(J)+1\big\}=\emptyset$, we understand $q_0=0$. Define $d_J(q):=\prod_{j=1}^{r(J)}{\rm gcd}(d_{J,j},q)$ and 
\[
\tilde{d}_J(q):=
\begin{cases}
d_J(q),& \text{ if } {\rm gcd}(q,d_{J,j})={\rm gcd}(q,\bar{d}_{J,j}) \text{ for all } j=1,\ldots,\bar{r}(J);\\
0,& \text{ otherwise},
\end{cases}
\]
with the convention $d_{J,\bar{r}(J)}=0$ if $\bar{r}(J)=r(J)+1$.

We now have enough tools to prove our first main result, which will be applied to graph coloring and flow problems in \autoref{Sec4}.
\begin{theorem}\label{Main1}
With the above notations, the counting formula of $M\big(\mathcal{A}_q^{(B,\bm b)}\big)$ is 
\begin{equation}\label{Counting-Formula}
\#M\big(\mathcal{A}_q^{(B,\bm b)}\big)=\sum_{J\subseteq[m]}(-1)^{\#J}\tilde{d}_J(q)q^{n-r(J)}.
\end{equation}
Moreover, there exists a quasi-polynomial $\chi^{\text{quasi}}\big(\mathcal{A}^{(B,\bm b)},q\big)$ with $\rho_C$ as a period such that  $\#M\big(\mathcal{A}^{(B,\bm b)}_q\big)=\chi^{\text{quasi}}\big(\mathcal{A}^{(B,\bm b)},q\big)$ for all integers $q>q_0$. More precisely,  $\chi^{\text{quasi}}\big(\mathcal{A}^{(B,\bm b)},q\big)$ has the form
\begin{equation}\label{CQP}
\chi^{\text{quasi}}\big(\mathcal{A}^{(B,\bm b)},q\big)=\sum_{J:\, J\subseteq[m],\,\bar{r}(J)=r(J)}(-1)^{\#J}\tilde{d}_J(q)q^{n-r(J)}.
\end{equation}
\end{theorem}
\begin{proof}
We apply the inclusion-exclusion principle to compute $M\big(\mathcal{A}_q^{(B,\bm b)}\big)$. For any subset $S\subseteq V_q(B,\bm b)$, its indicator function is defined as follows: $1_S(\bm x)=1$ if $\bm x\in S$, and 0 otherwise. Thus, for any $\bm x\in V_q(B,\bm b)$, $\bm x\in M\big(\mathcal{A}_q^{(B,\bm b)}\big)=V_q(B,\bm b)-\bigcup_{j=1}^mH_{j,q}^{(B,\bm b)}$ if and only if 
\[
\prod_{j=1}^m\big(1_{V_q(B,\bm b)}-1_{H_{j,q}^{(B,\bm b)}}\big)(\bm x)=1.
\]
Note that $1_{V_q(B,\bm b)}\cdot 1_{H_{j,q}^{(B,\bm b)}}=1_{H_{j,q}^{(B,\bm b)}}$ and $1_{H_{i,q}^{(B,\bm b)}}\cdot 1_{H_{j,q}^{(B,\bm b)}}=1_{H_{i,q}^{(B,\bm b)}\bigcap H_{j,q}^{(B,\bm b)}}$. It follows that 
\begin{align}\label{Eq1}
\#M\big(\mathcal{A}_q^{(B,\bm b)}\big)&=\sum_{\bm x\in V_q(B,\bm b)}\prod_{j=1}^m\big(1_{V_q(B,\bm b)}-1_{H_{j,q}^{(B,\bm b)}}\big)(\bm x)\notag\\
&=\sum_{\bm x\in V_q(B,\bm b)}\sum_{J\subseteq[m]}(-1)^{\#J}1_{\bigcap_{j\in J}H_{j,q}^{(B,\bm b)}}(\bm x)\notag\\
&=\sum_{J\subseteq[m]}(-1)^{\#J}\sum_{\bm x\in V_q(B,\bm b)}1_{\bigcap_{j\in J}H_{j,q}^{(B,\bm b)}}(\bm x)\notag\\
&=\sum_{J\subseteq[m]}(-1)^{\#J}\#\Big(\bigcap_{j\in J}H_{j,q}^{(B,\bm b)}\Big).
\end{align} 
Further notice that $\bigcap_{j\in J}H_{j,q}^{(B,\bm b)}$ is the solution set in $\mathbb{Z}_q^n$ of the system $[C_J]_q\bm x=[\bm c_J]_q$. From \autoref{Compute}, we have
\begin{equation}\label{Eq2}
\#\bigcap_{j\in J}H_{j,q}^{(B,\bm b)}=\begin{cases}q^{n-r(J)}\prod_{j=1}^{r(J)}{\rm gcd}(q,d_{J,j}),&\text{ if } [\bm c_J]_q\in\Col_q\big([C_J]_q\big);\\
0,&\text{ otherwise}.
\end{cases}
\end{equation}
Applying \autoref{Solvable} to \eqref{Eq2}, we deduce
\begin{equation}\label{Eq3}
\#\bigcap_{j\in J}H_{j,q}^{(B,\bm b)}=q^{n-r(J)}\tilde{d}_J(q).
\end{equation}
It follows that equations \eqref{Eq1} and \eqref{Eq3} directly yield \eqref{Counting-Formula} in \autoref{Main1}. 

Next, we prove the second part in \autoref{Main1}. Suppose $q>q_0$. If $J\subseteq[m]$ with $\bar{r}(J)=r(J)+1$, then $d_{J,\bar{r}(J)}=0$ and $0<{\rm gcd}(q,\bar{d}_{J,\bar{r}(J)})\le \bar{d}_{J,\bar{r}(J)}<q$ since $q>q_0\ge\bar{d}_{J,\bar{r}(J)}$. Consequently, ${\rm gcd}(q,d_{J,\bar{r}(J)})=q>{\rm gcd}(q,\bar{d}_{J,\bar{r}(J)})$. Hence, $\tilde{d}_J(q)=0$ in this case. It follows that \eqref{Counting-Formula} can be simplified to the following form
\begin{align}\label{Counting-Formula1}
\#M\big(\mathcal{A}_q^{(B,\bm b)}\big)=\sum_{J:\, J\subseteq[m],\,\bar{r}(J)=r(J)}(-1)^{\#J}\tilde{d}_J(q)q^{n-r(J)}.
\end{align}
We now consider the case $J\subseteq[m]$ with $\bar{r}(J)=r(J)$. Since $\rho_C={\rm lcm}\big(d_{J,r(J)}:J\subseteq[m],\emph{} \bar{r}(J)=r(J)\big)$ and $d_{J,j}\mid d_{J,r(J)}\mid\rho_C$ for all $j=1,\ldots,r(J)$, we have ${\rm gcd}(q+\rho_C,d_{J,j})={\rm gcd}(q,d_{J,j})$. In addition,  $\bar{d}_{J,\bar{r}(J)}\mid d_{J,r(J)}\mid \rho_C$ via \autoref{IDT}, hence, arguing as before, ${\rm gcd}(q+\rho_C,\bar{d}_{J,j})={\rm gcd}(q,\bar{d}_{J,j})$ for all $j=1,\ldots,r(J)$. Immediately, we deduce
\[
\tilde{d}_J(q+\rho_C)=\tilde{d}_J(q),\quad \forall\; J\subseteq[m]\text{ with }\bar{r}(J)=r(J).
\]
Combining \eqref{Counting-Formula1}, we conclude that the counting function $\#M\big(\mathcal{A}_q^{(B,\bm b)}\big)$ coincides with the quasi-polynomial $\chi^{\text{quasi}}\big(\mathcal{A}^{(B,\bm b)},q\big)$ with $\rho_C$ as a period for all integers $q>q_0$.
\end{proof}

It is worth noting that \autoref{Main1} recovers and extends the early results on characteristic quasi-polynomials of integral arrangements given in \cite{CW2012,KTT2008,KTT2011}. More specifically, taking $\bm a=\bm 0$ and $\bm b=\bm 0$,  \autoref{Main1} specializes to the central case given in \cite[Theorem 1.1]{CW2012}. When we restrict to the case $l = 0$, \autoref{Main1} reduces to \cite[Theorem 3.1]{KTT2011}. If, furthermore, $\bm{a} = \bm{0}$, it becomes exactly \cite[Theorem 2.4]{KTT2008}. 

As with \cite{KTT2011}, we refer to the quasi-period $\rho_C$ as the {\bf lcm period}, and term the quasi-polynomial $\chi^{\text{quasi}}\big(\mathcal{A}^{(B,\bm b)},q\big)$ the {\bf characteristic quasi-polynomial}. In this context, we can decompose $\chi^{\text{quasi}}\big(\mathcal{A}^{(B,\bm b)},q\big)$ into polynomial functions corresponding to residue classes of $\rho_C$. Namely, for every $j=1,\ldots,\rho_C$, there exists a unique polynomial $f_j\big(\mathcal{A}^{(B,\bm b)},t\big)$ such that
\[
f_j\big(\mathcal{A}^{(B,\bm b)},q\big)=\chi^{\text{quasi}}\big(\mathcal{A}^{(B,\bm b)},q\big) \;\quad\text{ for }\;q\equiv j \pmod{\rho_C}.
\]
In the next result, we examine the relationships existing among the polynomials $f_1,f_2,\ldots,f_{\rho_C}$, extending \cite[Theorem 2.3]{CW2012} to the non-central case. 
\begin{theorem}\label{Main2} With the above notations, the following results hold:
\begin{itemize}
\item[{\rm(1)}] The polynomial $f_a\big(\mathcal{A}^{(B,\bm b)},t\big)$ with $a=1,\ldots,\rho_C$, can be written in the form
\[
f_a\big(\mathcal{A}^{(B,\bm b)},t\big)=\sum_{J:\, J\subseteq[m],\,\bar{r}(J)=r(J)}(-1)^{\#J}\tilde{d}_J(a)t^{n-r(J)}.
\]
\item[{\rm(2)}] For any integers $1\le a,b\le\rho_C$, if ${\rm gcd}(a,\rho_C)={\rm gcd}(b,\rho_C)$, then
\[
f_a\big(\mathcal{A}^{(B,\bm b)},t\big)=f_b\big(\mathcal{A}^{(B,\bm b)},t\big).
\]
In other words, the characteristic quasi-polynomial $\chi^{\text{quasi}}\big(\mathcal{A}^{(B,\bm b)},q\big)$ has the gcd property.
\item[{\rm(3)}] For any integers $q>q_0$ and $1\le a\le\rho_C$, if ${\rm gcd}(q,\rho_C)=a$, then 
\[
\chi^{\text{quasi}}\big(\mathcal{A}^{(B,\bm b)},q\big)=f_a\big(\mathcal{A}^{(B,\bm b)},q\big).
\]
\end{itemize}
\end{theorem}
\begin{proof}
$(1)$ For every subset $J\subseteq[m]$ with $\bar{r}(J)=r(J)$ and $1\le j\le r(J)$, we have $0<d_{J,j}\mid d_{J,r(J)}\mid \rho_C$. If $q\equiv a\pmod{\rho_C}$, then $q\equiv a\pmod{d_{J,j}}$ for any $1\le j\le r(J)$. From the division algorithm, we deduce ${\rm gcd}(q,d_{J,j})={\rm gcd}(a,d_{J,j})$. This implies that $\tilde{d}_J(q)=\tilde{d}_J(a)$ when  $q\equiv a\pmod{\rho_C}$ and $\bar{r}(J)=r(J)$. From \eqref{CQP}, we have
\[
f_a\big(\mathcal{A}^{(B,\bm b)},q\big)=\chi^{\text{quasi}}\big(\mathcal{A}^{(B,\bm b)},q\big)=\sum_{J:\, J\subseteq[m],\,\bar{r}(J)=r(J)}(-1)^{\#J}\tilde{d}_J(a)q^{n-r(J)}.
\]
The above equality holds for infinitely many $q'>q_0$ with $q'\equiv a\pmod{\rho_C}$.  By the fundamental theorem of algebra, if two polynomials agree on infinitely many values, they must be identical. Therefore, we conclude $f_a\big(\mathcal{A}^{(B,\bm b)},t\big)=\sum_{J:\, J\subseteq[m],\,\bar{r}(J)=r(J)}(-1)^{\#J}\tilde{d}_J(a)t^{n-r(J)}$.

$(2)$ Note that ${\rm gcd}(q,d_{J,j})={\rm gcd}\big({\rm gcd}(q,\rho_C),d_{J,j}\big)$ since $d_{J,j}\mid\rho_C$. This means that $\tilde{d}_J(q)=\tilde{d}_J\big({\rm gcd}(q,\rho_C)\big)$ when $\bar{r}(J)=r(J)$. Together with ${\rm gcd}(a,\rho_C)={\rm gcd}(b,\rho_C)$, we deduce $\tilde{d}_J(a)=\tilde{d}_J(b)$ when $\bar{r}(J)=r(J)$. Combining part $(1)$, $f_a\big(\mathcal{A}^{(B,\bm b)},t\big)=f_b\big(\mathcal{A}^{(B,\bm b)},t\big)$ holds. 

$(3)$ If ${\rm gcd}(q,\rho_C)=a$, then ${\rm gcd}(q,\rho_C)={\rm gcd}(a,\rho_C)=a$. As shown in  the proof of part $(2)$, we have $\tilde{d}_J(q)=\tilde{d}_J(a)$ when $\bar{r}(J)=r(J)$. Hence,  $\chi^{\text{quasi}}\big(\mathcal{A}^{(B,\bm b)},q\big)=f_a\big(\mathcal{A}^{(B,\bm b)},q\big)$.
\end{proof}

Recall from \autoref{Sec1} that the subspace arrangement $\mathcal{A}^{(B,\bm b)}=\big\{H_1^{(B,\bm b)},\ldots,H_m^{(B,\bm b)}\big\}$ is defined in  the solution space $V(B,\bm b)=\big\{\bm x\in\mathbb{R}^n\mid B\bm x=\bm b\big\}$. When ${\rm rank}(B)\ne{\rm rank}(B,\bm b)$, $V(B,\bm b)=\big\{\bm x\in\mathbb{R}^n\mid B\bm x=\bm b\big\}$ is empty. Consequently, 
the characteristic polynomial $\chi\big(\mathcal{A}^{(B,\bm b)},t\big)$ is the zero polynomial. In this context, it is clear that 
\[
\chi\big(\mathcal{A}^{(B,\bm b)},q\big)=\chi^{\text{quasi}}\big(\mathcal{A}^{(B,\bm b)},q\big)=f_j\big(\mathcal{A}^{(B,\bm b)},q\big)=0
\]
for all integers $q>q_0$ and $j=1,\ldots,\rho_C$. When  ${\rm rank}(B)={\rm rank}(B,\bm b)$, the solution space $V(B,\bm b)$ is a nonempty affine subspace of $\mathbb{R}^n$. Then the characteristic polynomial of the truncated arrangement $\mathcal{A}^{(B,\bm b)}$ can be written in the form
\[
\chi\big(\mathcal{A}^{(B,\bm b)},t\big)=\sum_{J:\, J\subseteq[m],\,\bar{r}(J)=r(J)}(-1)^{\#J}t^{n-r(J)}.
\]
Assume $q>q_0$ is a positive integer coprime to $\rho_C$, i.e., ${\rm gcd}(q,\rho_C)=1$. Since $d_{J,j}\mid \rho_C$, we obtain ${\rm gcd}(q,d_{J,j})=1$ for all $J\subseteq[m]$ with $\bar{r}(J)=r(J)$ and $1\le j\le r(J)$. By the Interlacing Divisibility Theorem (\autoref{IDT}), we have $\bar{d}_{J,j}\mid d_{J,j}\mid \rho_C$. Hence, ${\rm gcd}(q,\tilde{d}_{J,j})=1$ for all $J\subseteq[m]$ with $\bar{r}(J)=r(J)$ and $1\le j\le r(J)$. Consequently, $\tilde{d}_J(q)=1$ for all $J\subseteq[m]$ with $\bar{r}(J)=r(J)$ when ${\rm gcd}(q,\rho_C)=1$. Substituting this into \eqref{CQP}, we arrive at
\[
\chi^{\text{quasi}}\big(\mathcal{A}^{(B,\bm b)},q\big)=\sum_{J:\, J\subseteq[m],\,\bar{r}(J)=r(J)}(-1)^{\#J}q^{n-r(J)}=\chi\big(\mathcal{A}^{(B,\bm b)},q\big).
\]
This implies that the characteristic polynomial $\chi\big(\mathcal{A}^{(B,\bm b)},t\big)$ can be computed by counting points in the complement of $\mathcal{A}^{(B,\bm b)}_q$ for all integers $q>q_0$ that are relatively prime to a constant $\rho_C$ depending only on $\mathcal{A}^{(B,\bm b)}$, which extends Athanasiadis's result \cite[Theorem 2.1]{Athanasiadis1999} to truncated arrangements. According to \autoref{Main2}, we further deduce $\chi\big(\mathcal{A}^{(B,\bm b)},t\big)=f_1\big(\mathcal{A}^{(B,\bm b)},t\big)$. We close this section by summarizing the above results as follows.
\begin{corollary}\label{Char-Quasi}
If ${\rm gcd}(q,\rho_C)=1$ with integer $q>q_0$, then
\[
\chi\big(\mathcal{A}^{(B,\bm b)},q\big)=\chi^{\text{quasi}}\big(\mathcal{A}^{(B,\bm b)},q\big).
\]
Moreover, we have 
\[
\chi\big(\mathcal{A}^{(B,\bm b)},t\big)=f_1\big(\mathcal{A}^{(B,\bm b)},t\big).
\]
\end{corollary}

\section{Uniform comparison}\label{Sec3}
This section establishes three unified comparisons: the unsigned coefficients of characteristic quasi-polynomials (see \autoref{Main3}); those of distinct constituents of a characteristic quasi-polynomial (see \autoref{Main4}); and the counting functions (see \autoref{Main5}).
\subsection{Comparison of coefficients}\label{Sec3-1}
In this subsection, we concentrate our attention on the first two of these comparisons. Throughout this section, we assume $r=r(\emptyset)={\rm rank}(B)$ and $s=r([m])={\rm rank}\begin{bmatrix}B\\A\end{bmatrix}$. From \eqref{CQP} in \autoref{Main1}, for all integers $q > q_0$, $\chi^{\text{quasi}}\big(\mathcal{A}^{(B,\bm{b})}, q\big)$ can be written in the form
\begin{equation}\label{CQP1}
\chi^{\text{quasi}}\big(\mathcal{A}^{(B,\bm b)},q\big)=\sum_{j=r}^s(-1)^{j-r}\beta_j(q)q^{n-j},
\end{equation}
where each coefficient $\beta_j(q)$ is a periodic function given by
\begin{equation}\label{Beta}
\beta_j(q):=(-1)^{j-r}\sum_{J:\, J\subseteq[m],\,\bar{r}(J)=r(J)=j}(-1)^{\#J}\tilde{d}_J(q).
\end{equation}

Inspired by Chen and Wang's work \cite{CW2012}, we examine the arithmetic properties of the coefficients $\beta_j(q)$ of $\chi^{\text{quasi}}\big(\mathcal{A}^{(B,\bm b)},q\big)$. For this purpose, the following lemmas are required.
\begin{lemma}\label{Invariant}
Let $U\in GL_n(\mathbb{Z})$ be a unimodular matrix, $A^*=AU$ and $B^*=BU$. Let $\mathcal{A}^*=\mathcal{A}(A^*,\bm a)$ denote the integral arrangement defined by the matrix $[A^*,\bm a]$, and $\mathcal{A}^{*(B^*,\bm b)}$  be the truncation of $\mathcal{A}^*$ by the matrix $[B^*,\bm b]$. Then, for all $q\in\mathbb{Z}_{>0}$, we have
\[
\#M\big(\mathcal{A}^{*(B^*,\bm b)},q\big)=\#M\big(\mathcal{A}^{(B,\bm b)},q\big).
\]
\end{lemma}
\begin{proof}
As with \eqref{Eq1}, we have 
\[
\#M\big(\mathcal{A}^{*(B^*,\bm b)},q\big)=\sum_{J\subseteq[m]}(-1)^{\#J}\#\Big(\bigcap_{j\in J}H_{j,q}^{(B^*,\bm b)}\Big).
\]
Comparing this and \eqref{Eq1}, it suffices to prove that
\[
 \#\Big(\bigcap_{j\in J}H_{j,q}^{(B^*,\bm b)}\Big)=\#\Big(\bigcap_{j\in J}H_{j,q}^{(B,\bm b)}\Big) \text{ for all } J\subseteq[m].
\]
Note that $\bigcap_{j\in J}H_{j,q}^{(B,\bm b)}$ is the solution set of $[C_J]_q\bm x=[\bm c_J]_q$. Thus, without loss of generality, we consider only the case $J=\emptyset$. The proof simplifies to verifying that
\[
\#V_q(B^*,\bm b)=\#V_q(B,\bm b).
\]
Notice that $[B,\bm b]\begin{bmatrix}U&\bm 0\\\bm 0&1\end{bmatrix}=[BU,\bm b]=[B^*,\bm b]$, and both matrices $U$ and $\begin{bmatrix}U&\bm 0\\\bm 0&1\end{bmatrix}$ are unimodular. This implies that both matrices $B$ and $B^*$ have identical invariant factors, and the augmented matrices $[B,\bm b]$ and $[B^*,\bm b]$ also share the same invariant factors. According to \autoref{Solvable}, we further deduce that $[\bm b]_q\in \Col_q(B)$ if and only if $[\bm b]_q\in \Col_q(B^*)$. Together with \autoref{Compute}, we obtain $\#V_q(B^*,\bm b)=\#V_q(B,\bm b)$. We finish the proof.
\end{proof}

Furthermore, we proceed to choose a unimodular matrix $U\in GL_n(\mathbb{Z})$ such that 
\[
\begin{bmatrix}B\\A\end{bmatrix}U=\begin{bmatrix}B^*\\A^*\end{bmatrix}=\begin{bmatrix}B^{\star}&\bm 0\\A^{\star}&\bm 0\end{bmatrix},
\]
where $\begin{bmatrix}B^{\star}\\A^{\star}\end{bmatrix}$ is a $(l+m)\times s$ integral matrix with rank $s$. Then the matrix $[A^{\star},\bm a]$ determines a hyperplane arrangement $\mathcal{A}^\star=\mathcal{A}(A^\star,\bm a)$ in $\mathbb{R}^s$. Analogously, we use $\mathcal{A}^{\star(B^\star,\bm b)}$ to represent the truncation of $\mathcal{A}^\star$ by the matrix $[B^\star,\bm b]$. Correspondingly, we have the following relations:
\[
V_q(B^*,\bm b)=V_q(B^\star,\bm b)\times \mathbb{Z}_q^{n-s},\quad H_{i,q}^*=H_{i,q}^\star\times\mathbb{Z}_q^{n-s},\quad H_{i,q}^{*(B^*,\bm b)}=H_{i,q}^{\star(B^\star,\bm b)}\times\mathbb{Z}_q^{n-s}.
\]
Immediately, we deduce
\[
V_q(B^*,\bm b)-\bigcup_{i=1}^mH_{i,q}^{*(B^*,\bm b)}=\Big(V_q(B^\star,\bm b)-\bigcup_{i=1}^mH_{i,q}^{\star(B^\star,\bm b)}\Big)\times\mathbb{Z}_q^{n-s}.
\]
It follows from \autoref{Invariant} that for all integers $q>q_0$,
\[
\chi^{\text{quasi}}\big(\mathcal{A}^{(B,\bm b)},q\big)=\chi^{\text{quasi}}\big(\mathcal{A}^{*(B^*,\bm b)},q\big)=\chi^{\text{quasi}}\big(\mathcal{A}^{\star(B^\star,\bm b)},q\big)q^{n-s}.
\]
When ${\rm rank}(B)={\rm rank}\begin{bmatrix}B\\ A\end{bmatrix}=s$, we immediately get
\[
\chi^{\text{quasi}}\big(\mathcal{A}^{(B,\bm b)},q\big)=\beta_s(q)q^{n-s}\quad\text{for all } q>q_0,
\]
where 
\begin{equation}\label{Coefficient}
\beta_s(q)=\sum_{J:\, J\subseteq[m],\,\bar{r}(J)=r(J)=s}(-1)^{\#J}\tilde{d}_J(q)=\chi^{\text{quasi}}\big(\mathcal{A}^{\star(B^\star,\bm b)},q\big).
\end{equation}
In the special case, we obtain the following key lemma.
\begin{lemma}\label{Key-Lemma}
Let ${\rm rank}(B)={\rm rank}\begin{bmatrix}B\\ A\end{bmatrix}=s$. Then $\chi^{\text{quasi}}\big(\mathcal{A}^{(B,\bm b)},q\big)=\beta_s(q)q^{n-s}$ for all integers $q>q_0$. Moreover, if integers $a,b>q_0$ with $a\mid b$ and $[\bm b]_b\in\Col_b(B)$, then $0\le\beta_s(a)\le\beta_s(b)$.
\end{lemma}
\begin{proof}
It remains to verify the second part. By \eqref{Coefficient}, without loss of generality, we may assume $s=n$. Thus, this proof reduces to proving that
\[
\beta_n(a)=\chi^{\text{quasi}}\big(\mathcal{A}^{(B,\bm b)},a\big)\le \chi^{\text{quasi}}\big(\mathcal{A}^{(B,\bm b)},b\big)=\beta_n(b).
\]
Since $a\mid b$, we write $b=ac$ for a positive integer $c$. Combining $[\bm b]_b\in\Col_b(B)$,  we have $[\bm b]_a\in\Col_a(B)$ by \autoref{Solvable}. It follows that there exists some $\bm x_0\in\mathbb{Z}^n$ such that $[B]_b[\bm x_0]_b=[\bm b]_b$. Clearly, $[\bm x_0]_a$ is also a solution of  $[B]_a\bm x=[\bm b]_a$. Now consider the new arrangement $\mathcal{A}^0=\mathcal{A}(A,\bm a-A\bm x_0)$ defined by the matrix $[A,\bm a-A\bm x_0]$ and its truncation $\mathcal{A}^{0(B,\bm b-B\bm x_0)}$  induced by the matrix $[B,\bm b-B\bm x_0]$. It is straightforward to check that for $q=a,b$, the following map
\[
\varphi_q:V_q(B,\bm b)\to V_q(B,\bm b-B\bm x_0),\quad [\bm x]_q\mapsto \varphi_q\big([\bm x]_q\big)=[\bm x-\bm x_0]_q
\]
is a bijection. Furthermore, it is easy to see that for $q=a,b$,
\[
\varphi_q\big(M(\mathcal{A}_q^{(B,\bm b)})\big)=M\big(\mathcal{A}_q^{0(B,\bm b-B\bm x_0)}\big).
\]
Note $[\bm b-B\bm x_0]_q=[\bm 0]_q$ for $q=a,b$. Hence, without loss of generality, we may assume $\bm b=\bm 0$. Hence, the proof further reduces to proving that
\[
\chi^{\text{quasi}}\big(\mathcal{A}^{(B,\bm 0)},a\big)\le \chi^{\text{quasi}}\big(\mathcal{A}^{(B,\bm 0)},b\big).
\]

Notice the fact that for any $\bm x\in\mathbb{Z}^n$, if $B\bm x\equiv\bm 0\pmod{a}$,  then $B(c\bm x)\equiv\bm 0\pmod{b}$ because $b=ac$. Naturally, we consider the following map
\[
\phi:V_a(B,\bm 0)\to V_b(B,\bm 0),\quad [\bm x]_a\mapsto\phi\big([\bm x]_a\big)=[c\bm x]_{b}.
\]
We now verify that $\phi$ is injective by contradiction. Suppose that there exist  two distinct elements $[\bm x]_a$ and  $[\bm y]_a$ in $V_a(B,\bm 0)$ such that $\phi\big([\bm x]_a\big)=\phi\big([\bm y]_a\big)$, i.e., $[c\bm x]_b=[c\bm y]_b$. Then $c(\bm x-\bm y)=b\bm z$ for some $\bm z\in\mathbb{Z}^n$. Substituting $b=ac$, we get $\bm x-\bm y=a\bm z$. Thus, $[\bm x]_a=[\bm y]_a$, which contradicts our initial assumption that $[\bm x]_a\ne[\bm y]_a$. Therefore, $\phi$ is an injection.

Next, we further show that the injective map $\phi$ induces an injection from $M\big(\mathcal{A}_a^{(B,\bm 0)}\big)$ to $M\big(\mathcal{A}_b^{(B,\bm 0)}\big)$. It remains to prove that $\phi\big([\bm x]_a\big)\in V_b(B,\bm 0)-\bigcup_{i=1}^mH_{i,b}^{(B,\bm 0)}$ whenever $[\bm x]_a\in V_a(B,\bm 0)-\bigcup_{i=1}^mH_{i,a}^{(B,\bm 0)}$. Arguing by contradiction, suppose there exists an element $[\bm x]_a\in V_a(B,\bm 0)-\bigcup_{i=1}^mH_{i,a}^{(B,\bm 0)}$ such that $\phi\big([\bm x]_a\big)=[c\bm x]_b\in H_{i,b}^{(B,\bm 0)}$ for some $i\in[m]$, with $H_i:\bm\alpha_i^T\bm x=a_i$ being the corresponding hyperplane in $\mathcal{A}$. We proceed by considering two cases: $a_i\equiv 0\pmod{b}$ and $a_i\not\equiv 0\pmod{b}$. For the former case, $\bm\alpha^T_i(c\bm x)=b\bm z=ac\bm z$ for some $\bm z\in\mathbb{Z}$ via $[c\bm x]_b\in H_{i,b}^{(B,\bm 0)}$ implies $\bm\alpha^T_i\bm x=a\bm z$. Consequently, $[\bm x]_a\in H_{i,a}^{(B,\bm 0)}$, a contradiction. For the latter case, $a_i\not\equiv 0\pmod{b}$ means
\[
\rank\begin{bmatrix}
B & \bm 0\\
c\bm\alpha^T_i & a_i
\end{bmatrix}
=\rank(B)+1=\rank\begin{bmatrix}B\\{c\bm\alpha^T_i}\end{bmatrix}+1=n+1.
\]
This means that $\bar{d}_{\{i\},n+1}\ne 0$, but $d_{\{i\},n+1}=0$. Since $b>q_0\ge \bar{d}_{\{i\},n+1}$, it further implies that 
$\begin{bmatrix}
B & \bm 0\\
c\bm\alpha^T_i & a_i
\end{bmatrix}_b$ 
has $n+1$ invariant factors, whereas  $\begin{bmatrix}B\\{c\bm\alpha^T_i}\end{bmatrix}_b$ has only $n$ invariant factors. It follows from \autoref{Solvable} that $\begin{bmatrix}
B\\ c\bm\alpha^T_i
\end{bmatrix}_b\bm y=\begin{bmatrix}
\bm 0\\ a_i
\end{bmatrix}_b$
has no solution. This contradicts the assumption that $[c\bm x]_b\in H_{i,b}^{(B,\bm 0)}$. In conclusion, $\phi$ indeed induces an injective map from $M(\mathcal{A}_a^{(B,\bm 0)})$ to $M(\mathcal{A}_b^{(B,\bm 0)})$. Consequently, we obtain  $\chi^{\text{quasi}}\big(\mathcal{A}^{(B,\bm 0)},a\big)\le \chi^{\text{quasi}}\big(\mathcal{A}^{(B,\bm 0)},b\big)$. We complete the proof.
\end{proof}

Recall that the integral arrangement $\mathcal{A}$ is defined by the matrix $[A,\bm a]$. Let $H$ be a hyperplane of $\mathcal{A}$ with the defining equation $\bm\alpha^T_i\bm x=a_i$, where $(\bm\alpha^T_i,a_i)$ is a row vector of the defining matrix $[A,\bm a]$ of $\mathcal{A}$. Let $[A',\bm a']$ denote the matrix obtained from $[A,\bm a]$ by deleting the row $(\bm\alpha^T_i,a_i)$, and let $[B',\bm b']$ denote the matrix obtained by adding the row $(\bm\alpha^T_i,a_i)$ to the bottom of $[B,\bm b]$. Then $[A',\bm\alpha']$ is the defining matrix of the integral arrangement $\mathcal{A}':=\mathcal{A}-\{H\}$. The arrangement $\mathcal{A}'$ is truncated by the matrices $[B,\bm b]$ and $[B',\bm b']$ to form the truncated arrangements $\mathcal{A}'^{(B,\bm b)}$ and $\mathcal{A}'^{(B',\bm b')}$, respectively. We observe the following deletion-restriction recurrence for the characteristic quasi-polynomial.
\begin{proposition}[Deletion-restriction formula] \label{DRF}For all integers $q>q_0$, with the above notations, we have
\[
\chi^{\text{quasi}}\big(\mathcal{A}^{(B,\bm b)},q\big)=\chi^{\text{quasi}}\big(\mathcal{A}'^{(B,\bm b)},q\big)-\chi^{\text{quasi}}\big(\mathcal{A}'^{(B',\bm b')},q\big).
\]
\end{proposition}
\begin{proof}
Without loss of generality, we assume $H=H_1$. We first prove that the complement of $\mathcal{A}'^{(B,\bm b)}$ can be decomposed into a disjoint union of the complement of $\mathcal{A}^{(B,\bm b)}$ and that of $\mathcal{A}'^{(B',\bm b')}$. Note that $\bm x\in M\big(\mathcal{A}_q'^{(B,\bm b)}\big)$ is equivalent to $\bm x\in V_q(B,\bm b)$ and $\bm x\notin H_{i,q}$ for all $i=2,\ldots,m$. This is further equivalent to either $\bm x\in V_q(B,\bm b)$ and $\bm x\notin H_{i,q}$ for all $i=1,2\ldots,m$, or $\bm x\in V_q(B,\bm b)$ and $\bm x\notin H_{i,q}$ for all $i=2\ldots,m$ but $\bm x\in H_{1,q}$. Consequently, we have
\[
\#M\big(\mathcal{A}'^{(B,\bm b)},q\big)=\#M\big(\mathcal{A}^{(B,\bm b)},q\big)+\#M\big(\mathcal{A}'^{(B',\bm b')},q\big).
\]
It follows from \autoref{Main1} that for integers $q>q_0$, $\#M\big(\mathcal{A}'^{(B',\bm b')},q\big)$ and $\#M\big(\mathcal{A}'^{(B,\bm b)},q\big)$ can be uniformly described by the quasi-polynomials $\chi^{\text{quasi}}\big(\mathcal{A}'^{(B',\bm b')},q\big)$ and $\chi^{\text{quasi}}\big(\mathcal{A}'^{(B,\bm b)},q\big)$ respectively,  which share the common period $\rho_C$. This implies that the result holds.
\end{proof}

To present our results conveniently and naturally, we introduce the concept of combinatorial equivalence for all positive integers. Two positive integers $a$ and $b$ are said to be {\bf combinatorially equivalent} with respect to $\mathcal{A}^{(B,\bm b)}$, denoted $a\stackrel{\mathcal{A}^{(B,\bm b)}}\sim b$, if 
\[
\bigcap_{j\in J}H_{j,a}^{(B,\bm b)}\ne\emptyset\iff \bigcap_{j\in J}H_{j,b}^{(B,\bm b)}\ne\emptyset,\quad\forall\, J\subseteq[m]\text{ with }\bar{r}(J)=r(J).
\]
Clearly, the relation $\stackrel{\mathcal{A}^{(B,\bm b)}}\sim$ is an equivalence relation on positive integers. The next lemma states that positive integers $a$ and $b$ are combinatorially equivalent whenever ${\rm gcd}(a,\rho_C)={\rm gcd}(b,\rho_C)$. This serves as a bridge connecting the combinatorial equivalence relation between positive integers and their greatest common divisor with $\rho_C$, as presented in \autoref{GCD2}.

\begin{lemma}\label{GCD1}
Let $a,b\in\mathbb{Z}_{>0}$. If ${\rm gcd}(a,\rho_C)={\rm gcd}(b,\rho_C)$, then $a\stackrel{\mathcal{A}^{(B,\bm b)}}\sim b$. 
\end{lemma}
\begin{proof}
Note that  ${\rm gcd}(d,\rho_C)=d$ for any positive integer $d$ dividing $\rho_C$. By ${\rm gcd}(a,\rho_C)={\rm gcd}(b,\rho_C)$, we have
\[
{\rm gcd}(a,d)={\rm gcd}\big({\rm gcd}(a,\rho_C),d\big)={\rm gcd}\big({\rm gcd}(b,\rho_C),d\big)={\rm gcd}(b,d).
\]
Therefore, for all $J\subseteq[m]$ with $\bar{r}(J)=r(J)$ and $1\le j\le r(J)$, we obtain 
\[
{\rm gcd}(a,d_{J,j})={\rm gcd}(b,d_{J,j})\quad\And\quad{\rm gcd}(a,\bar{d}_{J,j})={\rm gcd}(b,\bar{d}_{J,j}).
\]
It follows from \autoref{Solvable} that for any such $J$ and $q=a,b$,
\[
\bigcap_{j\in J}H_{j,q}^{(B,\bm b)}\neq\emptyset \iff {\rm gcd}(q,d_{J,j})={\rm gcd}(q,\bar{d}_{J,j})\text{ for all }1\le j\le r(J).
\]
Thus, $a$ and $b$ are combinatorially equivalent with respect to $\mathcal{A}^{(B,\bm b)}$.
\end{proof}
\begin{proposition}\label{GCD2}
Let $a,b\in\mathbb{Z}_{>0}$. Then $a\stackrel{\mathcal{A}^{(B,\bm b)}}\sim b$ if and only if  ${\rm gcd}(a,\rho_C)\stackrel{\mathcal{A}^{(B,\bm b)}}\sim {\rm gcd}(b,\rho_C)$.
\end{proposition}
\begin{proof}
As a direct result of \autoref{GCD1}, we have
\[
a\stackrel{\mathcal{A}^{(B,\bm b)}}\sim {\rm gcd}(a,\rho_C)\quad\And\quad b\stackrel{\mathcal{A}^{(B,\bm b)}}\sim {\rm gcd}(b,\rho_C).
\]
Since $\stackrel{\mathcal{A}^{(B,\bm b)}}\sim$ is an equivalence relation, the result follows immediately.
\end{proof}

After collecting the necessary concepts and lemmas, we now state our main result. The following theorem establishes a unified comparison between the unsigned coefficients of characteristic quasi-polynomials based on the divisibility relation.
\begin{theorem}\label{Main3}
For any integers $a,b>q_0$, if ${\rm gcd}(a,\rho_C)\mid{\rm gcd}(b,\rho_C)$ and $a\stackrel{\mathcal{A}^{(B,\bm b)}}\sim b$, then 
\begin{equation}\label{eq:compare beta}
0\le \beta_j(a)\le\beta_j(b),\quad\forall\, r\le j\le s.
\end{equation}
Moreover,  when $\bm a=\bm 0$ and $\bm b=\bm 0$, for all $a,b\in\mathbb{Z}_{>0}$, if ${\rm gcd}(a,\rho_C)\mid{\rm gcd}(b,\rho_C)$, then
\[
0\le \beta_j(a)\le\beta_j(b),\quad\forall\, r\le j\le s.
\]
\end{theorem}
\begin{proof}
Let $c={\rm gcd}(b,\rho_C)/{\rm gcd}(a,\rho_C)\in\mathbb{Z}_{>0}$, $a'={\rm gcd}(a,\rho_C)+q_0\rho_C$ and $b'=ca'$. Then
\[
{\rm gcd}(a',\rho_C)={\rm gcd}(a,\rho_C)\quad\And\quad {\rm gcd}(b',\rho_C)={\rm gcd}(b,\rho_C).
\]
Applying \autoref{GCD2} to $a\stackrel{\mathcal{A}^{(B,\bm b)}}\sim b$, we obtain $a'\stackrel{\mathcal{A}^{(B,\bm b)}}\sim b'$. From \autoref{Main2}, for all $J\subseteq[m]$ with $\bar{r}(J)=r(J)$, we have
\[
\tilde{d}_J(a')=\tilde{d}_J(a)\quad\And\quad \tilde{d}_J(b')=\tilde{d}_J(b).
\]
By the definition of $\beta_j(q)$ in \eqref{Beta}, we deduce that 
\[
\beta_j(a')=\beta_j(a)\quad\And\quad\beta_j(b')=\beta_j(b)\,\quad\forall\,\,r\le j\le s.
\]
Therefore, the proof boils down to showing that $0\le \beta_j(a')\le\beta_j(b')$ for all $r\le j\le s$ when $b'=ca'$ and $a'\stackrel{\mathcal{A}^{(B,\bm b)}}\sim b'$.

We now proceed by induction on $m$. If $m=0$, then $s=r$. Note $b'=ca'$. When $[\bm b]_{b'}\in\Col_{b'}(B)$, \autoref{Key-Lemma} directly implies $0\le\beta_s(a')\le\beta_s(b')$. If $[\bm b]_{b'}\notin\Col_{b'}(B)$, then $V_{b'}(B,\bm b)=\emptyset$ via \autoref{Solvable}. Thus $V_{a'}(B,\bm b)=\emptyset$ from $a'\stackrel{\mathcal{A}^{(B,\bm b)}}\sim b'$. Consequently, $\beta_s(a')=\beta_s(b')=0$ in this case. 

When $m\ge 1$, it suffices to consider two cases: $r=s$ and $r<s$. For the former case, the result follows directly from \autoref{Key-Lemma}. For the latter case, there exists a row $\bm\alpha^T_i$ of $A$ that is linearly independent of all rows of $B$. Let $(\bm\alpha^T_i,a_i)$ be its counterpart in $[A,\bm a]$. Let $[A',\bm a']$ denote the matrix obtained from $[A,\bm a]$ by deleting the row $(\bm\alpha^T_i,a_i)$, and $[B',\bm b']$ denote the matrix obtained by adding the row $[\bm\alpha^T_i,a_i]$ to the bottom of $[B,\bm b]$. Then
\[
\rank(B')=r+1,\quad\rank\begin{bmatrix}
B\\A'
\end{bmatrix}\le s\quad\And\quad 
\rank\begin{bmatrix}
B'\\A'
\end{bmatrix}=s.
\]
Let $H$ be a hyperplane of $\mathcal{A}$ with the defining equation $\bm\alpha^T_i\bm x=a_i$. Then the integral arrangement $\mathcal{A}'=\mathcal{A}-\{H\}$ has the defining matrix $[A',\bm a']$. The deletion-restriction formula in \autoref{DRF} means that
\[
\chi^{\text{quasi}}\big(\mathcal{A}^{(B,\bm b)},q\big)=\chi^{\text{quasi}}\big(\mathcal{A}'^{(B,\bm b)},q\big)-\chi^{\text{quasi}}\big(\mathcal{A}'^{(B',\bm b')},q\big).
\]
Clearly, $a'\stackrel{\mathcal{A}'^{(B,\bm b)}}\sim b'$ and $a'\stackrel{\mathcal{A}'^{(B',\bm b')}}\sim b'$. As $\mathcal{A}'$ has fewer hyperplanes than $\mathcal{A}$, the induction hypothesis implies that
\[
\chi^{\text{quasi}}\big(\mathcal{A}'^{(B,\bm b)},q\big)=\sum_{j=r}^{s}(-1)^{j-r}\beta'_j(q)q^{n-j}\quad\text{and}\quad \chi^{\text{quasi}}\big(\mathcal{A}'^{(B',\bm b')},q\big)=\sum_{j=r+1}^{s}(-1)^{j-r-1}\beta^{''}_j(q)q^{n-j}.
\]
satisfy
\[
0\le\beta'_j(a')\le\beta'_j(b')\quad\text{ and }\quad 0\le\beta^{''}_j(a')\le\beta^{''}_j(b').
\]
It follows that 
\[
0\le\beta_j(a')=\beta'_j(a')+\beta^{''}_j(a')\le\beta'_j(b')+\beta^{''}_j(b')=\beta_j(b')
\]
holds for all $r\le j\le s$ in this case. 

When $\bm a=\bm 0$ and $\bm b=\bm 0$, it is clear that $q_0=0$ and $a$ is combinatorially equivalent to $b$ with respect to $\mathcal{A}^{(B,\bm 0)}$ for any $a,b\in\mathbb{Z}_{>0}$. Then the result follows directly from \eqref{eq:compare beta}. 
\end{proof}

In fact, the first paragraph in the proof of \autoref{Main3} implicitly implies the following property: for any integers $a,b>q_0$, if ${\rm gcd}(a,\rho_C)\mid{\rm gcd}(b,\rho_C)$, then 
\begin{equation}\label{Equi}
a\stackrel{\mathcal{A}^{(B,\bm b)}}\sim b\iff [\bm c_J]_b\in\Col_b(C_J) \text{ whenever } [\bm c_J]_a\in\Col_a(C_J) \text{ for } J\subseteq[m].
\end{equation}
Specifically, let $a'$ and $b'$ be as defined in the proof of \autoref{Main3}. Since $a' \mid b'$, it follows from \autoref{Solvable} that $\bigcap_{j\in J}H_{j,a'}^{(B,\bm b)}\ne\emptyset$ whenever $\bigcap_{j\in J}H_{j,b'}^{(B,\bm b)}\ne\emptyset$ is equivalent to $[\bm c_J]_{a'}\in\Col_{a'}(C_J)$ when $[\bm c_J]_{b'}\in\Col_{b'}(C_J)$. Additionally, it is straightforward to check that for any $J\subseteq [m]$ with $\bar{r}(J)=r(J)$,  
\[
[\bm c_J]_a\in\Col_a(C_J)\iff [\bm c_J]_{a'}\in\Col_{a'}(C_J)\quad \text{and} \quad [\bm c_J]_b\in\Col_b(C_J)\iff [\bm c_J]_{b'}\in\Col_{b'}(C_J).
\]
Consequently, ${\rm gcd}(a,\rho_C)\mid{\rm gcd}(b,\rho_C)$ deduces that $\bigcap_{j\in J}H_{j,a}^{(B,\bm b)}\ne\emptyset$ whenever $\bigcap_{j\in J}H_{j,b}^{(B,\bm b)}\ne\emptyset$ for any $J\subseteq [m]$ with $\bar{r}(J)=r(J)$. Applying \autoref{Solvable} again, we conclude that the equivalence statement in \eqref{Equi} holds.

Furthermore, the condition that $[\bm c_J]_b\in\Col_b(C_J)$ whenever $[\bm c_J]_a\in\Col_a(C_J)$ for any $J\subseteq[m]$ with $\bar{r}(J)=r(J)$ in \autoref{Main3} is necessary. For example, let
\[
B=\begin{bmatrix}
2 & 0
\end{bmatrix}
,\quad \bm b=\begin{bmatrix}
0
\end{bmatrix},\quad A=\begin{bmatrix}
2 & 2
\end{bmatrix},\quad \bm a=\begin{bmatrix}
3
\end{bmatrix}.
\]
Then
\[
V_q(B,\bm b)=\big\{(x_1,x_2)^T\in\mathbb{Z}_q^2:[2]_qx_1=[0]_q\big\}\quad\And\quad\mathcal{A}=\mathcal{A}(A,\bm a)=\{H:2x_1+2x_2=3\}.
\]
Thus $\mathcal{A}_q^{(B,\bm b)}=\big\{H_q^{(B,\bm b)}=H_q\cap V_q(B,\bm b)\big\}$. It is obvious that $\rho_C=2$ and $q_0=2$. Let integers $q>q_0=2$. Then the characteristic quasi-polynomial $\chi^{\text{quasi}}\big(\mathcal{A}^{(B,\bm b)},q\big)$ is
\[
\chi^{\text{quasi}}\big(\mathcal{A}^{(B,\bm b)},q\big)=\begin{cases}2q,&\text{ if }q\text{ is even};\\
q-1,&\text{ if }q\text{ is odd}.\end{cases}
\]
We now consider $a=3$ and $b=6$. Clearly, ${\rm gcd}(a,\rho_C)\mid{\rm gcd}(b,\rho_C)$, and $H_3\cap V_3(B,\bm b)\ne\emptyset$, but $H_6\cap V_6(B,\bm b)=\emptyset$. By outline analysis, we have
\[
\beta_1(3)=1<2=\beta_1(6) \quad\And\quad \beta_2(3)=1>0= \beta_2(6)  .
\]
This means that \eqref{eq:compare beta} in \autoref{Main3} may fail to hold when the condition that $[\bm c_J]_b\in\Col_b(C_J)$ whenever $[\bm c_J]_a\in\Col_a(C_J)$ for $J\subseteq[m]$ is moved.

Recall from part (1) in \autoref{Main2} that each constituent $f_a\big(\mathcal{A}^{(B,\bm b)},t\big)$ ($a=1,\ldots,\rho_C$) of the characteristic quasi-polynomial $\chi^{\text{quasi}}\big(\mathcal{A}^{(B,\bm b)},q\big)$ has the following form
\[
f_a\big(\mathcal{A}^{(B,\bm b)},t\big)=\sum_{J:\, J\subseteq[m],\,\bar{r}(J)=r(J)}(-1)^{\#J}\tilde{d}_J(a)t^{n-r(J)}.
\]
Let
\begin{equation}\label{gamma}
\gamma_j(a):=(-1)^{j-r}\sum_{J:\, J\subseteq[m],\,\bar{r}(J)=r(J)=j}(-1)^{\#J}\tilde{d}_J(a).
\end{equation}
Then the polynomial $f_a\big(\mathcal{A}^{(B,\bm b)},t\big)$ can further be written the form
\[
f_a\big(\mathcal{A}^{(B,\bm b)},t\big)=\sum_{j=r}^s(-1)^{j-r}\gamma_j(a)t^{n-j}.
\]
Via part (3) of \autoref{Main2}, for any integers $q>q_0$ with ${\rm gcd}(q,\rho_C)={\rm gcd}(a,\rho_C)$, we have
\[
\chi^{\text{quasi}}\big(\mathcal{A}^{(B,\bm b)},q\big)=f_a\big(\mathcal{A}^{(B,\bm b)},q\big).
\]
By further comparing \eqref{Beta} and \eqref{gamma}, we readily observe that
\begin{equation}\label{Beta-gamma}
\beta_j(q)=\gamma_j(a),\quad\forall\, r\le j\le s.
\end{equation}
Therefore, building on \autoref{Main3}, we derive a unified comparison relation between the unsigned coefficients $\gamma_j(a)$ of  distinct constituents $f_a\big(\mathcal{A}^{(B,\bm b)},t\big)$.
\begin{theorem}\label{Main4}
For any integers $1\le a,b\le\rho_C$, if ${\rm gcd}(a,\rho_C)\mid{\rm gcd}(b,\rho_C)$ and  ${\rm gcd}(a,\rho_C)\stackrel{\mathcal{A}^{(B,\bm b)}}\sim {\rm gcd}(b,\rho_C)$, then 
\[
0\le \gamma_j(a)\le\gamma_j(b),\quad\forall\, r\le j\le s.
\]
Moreover, when $\bm a=\bm 0, \bm b=\bm 0$, for any integers $1\le a,b\le\rho_C$,  if ${\rm gcd}(a,\rho_C)\mid{\rm gcd}(b,\rho_C)$, then 
\[
0\le \gamma_j(a)\le\gamma_j(b),\quad\forall\, r\le j\le s.
\]
\end{theorem}
\begin{proof}
Let $c={\rm gcd}(b,\rho_C)/{\rm gcd}(a,\rho_C)\in\mathbb{Z}_{>0}$, $a'={\rm gcd}(a,\rho_C)+q_0\rho_C>q_0$ and $b'=ca'>q_0$. As proved in the first paragraph of the proof of \autoref{Main3}, we have
\[
{\rm gcd}(a',\rho_C)={\rm gcd}(a,\rho_C)\mid {\rm gcd}(b,\rho_C)={\rm gcd}(b',\rho_C)\quad\And\quad a'\stackrel{\mathcal{A}^{(B,\bm b)}}\sim b'.
\]
Immediately, from \eqref{Beta-gamma}, we have
\[
\beta_j(a')=\gamma_j(a)\quad\And\quad \beta_j(b')=\gamma_j(b).
\]
Applying \autoref{Main3}, we conclude that $0\le\gamma_j(a)\le\gamma_j(b)$ for all $r\le j\le s$. By the same reasoning, the result holds for the central case.
\end{proof}

Furthermore, if $H_1^{(B,\bm b)},\ldots,H_m^{(B,\bm b)}$ are codimension-one subspaces of $V(B,\bm b)$, the celebrated broken circuit theorem (as stated in  \cite[p.435]{Stanley2007}) asserts that the characteristic polynomial $\chi\big(\mathcal{A}^{(B,\bm b)},t\big)$ can be expressed as
\[
\chi\big(\mathcal{A}^{(B,\bm b)},t\big)=\sum_{j=r}^s(-1)^{j-r}w_j\big(\mathcal{A}^{(B,\bm b)}\big)t^{n-j}
\]
where $w_j\big(\mathcal{A}^{(B,\bm b)}\big)>0$ denotes the number of $j$-subsets of $\mathcal{A}^{(B,\bm b)}$ containing no broken circuits. 

The following corollary first gives a uniform comparison between the coefficients of the characteristic polynomial $\chi\big(\mathcal{A}^{(B,\bm b)},t\big)$ and those of the polynomial  $f_a\big(\mathcal{A}^{(B,\bm b)},t\big)$. This further implies that, under suitable conditions, all coefficients of $f_a\big(\mathcal{A}^{(B,\bm b)},t\big)$ are nonzero and alternate in sign.
\begin{corollary}\label{Nonzero-Alternate}
Suppose $V(B,\bm b)\ne\emptyset$ and $H_1^{(B,\bm b)},\ldots,H_m^{(B,\bm b)}$ are codimension-one subspaces of $V(B,\bm b)$. For any integers $q>q_0$ and $1\le a\le\rho_C$, if ${\rm gcd}(q,\rho_C)={\rm gcd}(a,\rho_C)$ and $[\bm c_J]_q\in\Col_q(C_J)$ for all $J\subseteq[m]$ with $\bar{r}(J)=r(J)$, then 
\[
0<w_j\big(\mathcal{A}^{(B,\bm b)}\big)=\gamma_j(1)\le\gamma_j(a)=\beta_j(q),\quad\forall\, r\le j\le s.
\]
Moreover, when $\bm a=\bm 0$ and $\bm b=\bm 0$, for any integers $q\in\mathbb{Z}_{>0}$ and $1\le a\le\rho_C$, if ${\rm gcd}(q,\rho_C)={\rm gcd}(a,\rho_C)$, then
\[
0<w_j\big(\mathcal{A}^{(B,\bm 0)}\big)=\gamma_j(1)\le\gamma_j(a)=\beta_j(q),\quad\forall\, r\le j\le s.
\]
\end{corollary}
\begin{proof}
Notice that $w_j\big(\mathcal{A}^{(B,\bm b)}\big)=\gamma_j(1)$ via \autoref{Char-Quasi} and $\gamma_j(a)=\beta_j(q)$ by \eqref{Beta-gamma}. The proof reduces to proving $\gamma_j(1)\le\gamma_j(a)$. Consider $b=1+q_0\rho_C>q_0$. Then we have
\[
{\rm gcd}(b,\rho_C)={\rm gcd}(1,\rho_C)=1\mid {\rm gcd}(a,\rho_C)={\rm gcd}(q,\rho_C).
\]
Additionally, the preceding discussion of \autoref{Char-Quasi} shows that $[\bm c_J]_b\in\Col_b(C_J)$ for all $J\subseteq[m]$ with $\bar{r}(J)=r(J)$ when ${\rm gcd}(b,\rho_C)=1$. Together with $[\bm c_J]_q\in\Col_q(C_J)$ for every $J\subseteq[m]$ satisfying $\bar{r}(J)=r(J)$, we deduce from \autoref{GCD1} and \eqref{Equi} that
\[
1\stackrel{\mathcal{A}^{(B,\bm b)}}\sim b\stackrel{\mathcal{A}^{(B,\bm b)}}\sim q\stackrel{\mathcal{A}^{(B,\bm b)}}\sim a.
\]
Applying \autoref{Main4}, we obtain $\gamma_j(1)\le\gamma_j(a)$. The same reasoning yields the conclusion for the central case. We finish the proof.
\end{proof}

When $l=0$, we have $\mathcal{A}=\mathcal{A}^{(B,\bm b)}$. In this context, \eqref{CQP1} gives the expression for the characteristic quasi-polynomial $\chi^{\text{quasi}}(\mathcal{A},q)$ as
\[
\chi^{\text{quasi}}\big(\mathcal{A},q\big)=\sum_{j=0}^s(-1)^j\beta_j(q)q^{n-j},
\]
where $s$ is the rank of $[A,\bm a]$. As an immediate consequence of \eqref{Beta-gamma}, \autoref{Main4} and \autoref{Nonzero-Alternate}, the coefficients of the polynomial $f_a(\mathcal{A},t)$ are nonzero and alternate in sign whenever $a$ is combinatorially equivalent $1$ with respect to $\mathcal{A}$ and its unsigned coefficients are no smaller than those of the characteristic polynomial $\chi(\mathcal{A},t)$.
\begin{corollary}\label{Nonzero-Alternate1}
For any integers $q>q_0$ and $1\le a,b\le\rho_C$, if ${\rm gcd}(a,\rho_C)\mid {\rm gcd}(b,\rho_C)={\rm gcd}(q,\rho_C)$, ${\rm gcd}(a,\rho_C)\stackrel{\mathcal{A}}\sim {\rm gcd}(b,\rho_C)$ and $[\bm a_J]_q\in\Col_q(A_J)$ for all $J\subseteq[m]$ with $\bar{r}(J)=r(J)$, then 
\[
0<w_j(\mathcal{A})=\gamma_j(1)\le\gamma_j(a)\le\gamma_j(b)=\beta_j(q),\quad\forall\, 0\le j\le s.
\]
Moreover, when $\bm a=\bm 0$, for any integers $q\in\mathbb{Z}_{>0}$ and $1\le a,b\le\rho_C$, if ${\rm gcd}(a,\rho_C)\mid {\rm gcd}(b,\rho_C)={\rm gcd}(q,\rho_C)$, then
\[
0<w_j(\mathcal{A})=\gamma_j(1)\le\gamma_j(a)\le\gamma_j(b)=\beta_j(q),\quad\forall\, 0\le j\le s.
\]
\end{corollary}
As explained earlier, the motivation traces back to Chen and Wang's work \cite{CW2012}. At the end of this section, we clarify how our results are connected to theirs. It is worth remarking that whenever positive integers $a$ divides $b$, the divisibility ${\rm gcd}(a,\rho_C)\mid {\rm gcd}(b,\rho_C)$ is automatically satisfied. In the central case: $\bm a=\bm 0$ and $\bm b=\bm 0$, it is clear that $q_0=0$ and positive integers $a$ and $b$ are combinatorially equivalent with respect to $\mathcal{A}^{(B,\bm 0)}$. Therefore, as a special case, we can recover Chen and Wang's main results for the central case by \autoref{Main3}, \autoref{Nonzero-Alternate} and \autoref{Nonzero-Alternate1}. However, ${\rm gcd}(a,\rho_C)\mid {\rm gcd}(b,\rho_C)$ may still hold even when $a\nmid b$. Thus, our work not only extends Chen and Wang's results to the non-central case, but also weakens their original hypothesis to a strictly more general one. 
\begin{theorem}[\cite{CW2012}, Theorem 1.2 and Corollary 1.3]
When $\bm a=\bm 0,\bm b=\bm 0$, the following results hold:
\begin{itemize}
\item[{\rm (1)}] For all $a,b\in\mathbb{Z}_{>0}$ with $a\mid b$, $0\le\beta_j(a)\le\beta_j(b)$ for every $r\le j\le s$.
\item[{\rm (2)}] If $H_1^{(B,\bm 0)},\ldots,H_m^{(B,\bm 0)}$ are codimension-one subspaces of $V(B,\bm 0)$, then $0<w_j\big(\mathcal{A}^{(B,\bm 0)}\big)\le\beta_j(q)$ for all integers $q\in\mathbb{Z}_{>0}$ and $r\le j\le s$.
\item[{\rm (3)}] If $l=0$, then $0<w_j(\mathcal{A})\le\beta_j(q)$ for all integers $q\in\mathbb{Z}_{>0}$ and $0\le j\le s$.
\end{itemize}
\end{theorem}
\subsection{Comparison of counting functions}\label{Sec3-2}
In this subsection, we restrict our attention to the interrelation among the cardinalities of $M\big(\mathcal{A}_q^{(B,\bm b)}\big)$ for $q\in\mathbb{Z}_{>0}$. To this end, for any $p,q\in\mathbb{Z}_{>0}$, we begin by considering a map defined by
\[
\pi_q:\mathbb{Z}^n_{pq}\to\mathbb{Z}^n_q,\quad[\bm x]_{pq}\mapsto\pi_q\big([\bm x]_{pq}\big)=[\bm x]_q\quad\text{where}\quad\bm x\in\mathbb{Z}^n.
\]
For any $\bm x_1, \bm x_2\in\mathbb{Z}^n$, if $\bm x_1\equiv\bm x_2\pmod{pq}$, then $\bm x_1\equiv\bm x_2\pmod q$ necessarily holds. Thus, $\pi_q$ is well-defined. Moreover, the following inclusion holds: for any $M\in\mathcal{M}_{m\times n}(\mathbb{Z})$ and $\bm c\in\mathbb{Z}^m$, the image of  $V_{pq}(M,\bm c)$ under the restricted map $\pi_q|_{V_{pq}(M,\bm c)}$ is contained in $V_q(M,\bm c)$, i.e.,
\[
\pi_q\big(V_{pq}(M, \bm c)\big)\subseteq V_q(M, \bm c).
\]
However, for general $p,q\in\mathbb{Z}_{>0}$, the restriction $\pi_q|_{V_{pq}(M, \bm c)}$ is not necessarily surjective. Subsequently, we give a necessary and sufficient condition for the restriction $\pi_q|_{V_{pq}(M, \bm c)}$ to be surjective.
\begin{lemma}\label{lem:surjective modular map}
Let $M\in\mathcal{M}_{m\times n}(\mathbb{Z})$, $\bm c\in\mathbb{Z}^m$, and $p,q\in\mathbb{Z}_{>0}$ such that  $[\bm c]_{pq}\in\Col_{pq}(M)$. The restriction of $\pi_q$ to $V_{pq}(M,\bm c)$ is a surjective map from $V_{pq}(M,\bm c)$ to $V_q(M,\bm c)$ if and only if ${\rm gcd}(p,d_i){\rm gcd}(q,d_i)\mid d_i$ for all $1\le i\le r$, where $d_1\mid\cdots\mid d_r$ denote the invariant factors of $M$.
\end{lemma}
\begin{proof}
Since $[\bm c]_{pq}\in\Col_{pq}(M)$, there exists some $\bm x_0\in\mathbb{Z}^n$ such that $[M]_{pq}[\bm x_0]_{pq}=[\bm c]_{pq}$. Immediatelly, $[M]_q[\bm x_0]_q=[\bm c]_q$ holds as well.  We now introduce two translation maps associated with $\bm x_0$ as follows:
\[
\varphi_a:\mathbb{Z}^n_a\to\mathbb{Z}^n_a,\quad [\bm x]_a\mapsto \varphi_a\big([\bm x]_a\big)=[\bm x-\bm x_0]_a, \quad\text{for } a=pq, q.
\]
It is obvious that both maps $\varphi_{pq}$ and $\varphi_q$ are well-defined. Furthermore, the following commutative diagram holds:
\[
\begin{tikzcd}
V_{pq}(M,\bm c) \arrow{r}{\pi_q} \arrow{d}{\varphi_{pq}} & V_q(M,\bm c) \arrow{d}{\varphi_q} \\
V_{pq}(M,\bm 0) \arrow{r}{\pi_q} & V_q(M,\bm 0)
\end{tikzcd}.
\]
Notice that for each $a=pq,q$, the restriction of $\varphi_a$ to $V_a(M, \bm c)$ induces a bijection between $V_a(M, \bm c)$ and $V_a(M, \bm 0)$. Then the restriction $\pi_q|_{V_{pq}(M,\bm c)}:V_{pq}(M,\bm c)\to V_q(M,\bm c)$ is surjective if and only if the restriction $\pi_q|_{V_{pq}(M,\bm 0)}:V_{pq}(M,\bm 0)\to V_q(M,\bm 0)$ is surjective. Therefore, without loss of generality, we may assume $\bm c=\bm 0$.

We now turn to the proof in the special case $\bm c=\bm 0$. From the theory of elementary divisors, there exist unimodular matrices $P\in GL_m(\mathbb{Z})$ and $Q\in GL_n(\mathbb{Z})$ such that
\[
PMQ=D={\rm diag}(d_1\ldots,d_r,0,\ldots,0).
\]
It follows that $V_a(M,\bm 0)=V_a(PM,\bm 0)$ for $a=pq,q$. Since $Q$ is an invertible matrix, the following group homomorphisms
\[
\phi_a: V_a(M,\bm 0)=V_a(PM,\bm 0)\to V_a(D,\bm 0),\quad\bm x\mapsto Q^{-1}\bm x,\quad\text{for }a=pq, q
\]
are bijective. Then the following commutative diagram holds:
\[
\begin{tikzcd}
V_{pq}(M,\bm 0) \arrow{r}{\pi_q} \arrow{d}{\phi_{pq}} & V_q(M,\bm 0) \arrow{d}{\phi_q} \\
V_{pq}(D,\bm 0) \arrow{r}{\pi_q} & V_q(D,\bm 0)
\end{tikzcd}.
\]
This implies that the restriction $\pi_q|_{V_{pq}(M,\bm 0)}:V_{pq}(M,\bm 0)\to V_q(M,\bm 0)$ is surjective if and only if the restriction $\pi_q|_{V_{pq}(D,\bm 0)}:V_{pq}(D,\bm 0)\to V_q(D,\bm 0)$ is surjective. This is further equivalent to the statement  that for any $[\bm v]_q\in V_q(D,\bm 0)$ with $\bm v=(v_1,\ldots,v_n)^T\in\mathbb{Z}^n$, there exists $\bm u=(u_1,\ldots,u_n)^T\in\mathbb{Z}^n$ such that
\[
D\bm u\equiv\bm 0\pmod{pq}\quad\And\quad \bm u\equiv\bm v\pmod q.
\]
Equivalently, for any $\bm v\in\mathbb{Z}^n$ satisfying $q\mid d_iv_i$ for $i=1,\dots,r$, there exist integers $k_1,\dots,k_r$ such that
\[
d_i(v_i-k_iq)\equiv0\pmod{pq},\quad i=1,\dots,r.
\]
By number theory, the existence of such integers $k_i$ is equivalent to the condition that ${\rm gcd}(d_iq,pq)\mid d_iv_i$ for all $1\le i\le r$. Alternatively, for any $\bm v\in\mathbb{Z}^n$, if $\frac{q}{{\rm gcd}(q,d_i)}\mid v_i$ for each $i=1,\ldots,r$, then  $\frac{q{\rm gcd}(p,d_i)}{{\rm gcd}(pq,d_i)}\,\Big|\,v_i$ holds for all $1\le i\le r$. From number theory again, the above statement holds precisely when
\[
\frac{q{\rm gcd}(p,d_i)}{{\rm gcd}(pq,d_i)}\Big|\frac{q}{{\rm gcd}(q,d_i)}\Longleftrightarrow {\rm gcd}(p,d_i){\rm gcd}(q,d_i)\mid{\rm gcd}(pq,d_i) \quad(i=1,\ldots,r).
\]
It follows from ${\rm gcd}(p,d_i){\rm gcd}(q,d_i)|pq$ that the above divisibility relation is further equivalent to ${\rm gcd}(p,d_i){\rm gcd}(q,d_i)\mid d_i$ for all $1\le i\le r$. We complete the proof.
\end{proof}

With the preceding lemma established, we now turn our attention to the desired result. Remarkably, \autoref{lem:surjective modular map} indicates that the hypothesis ${\rm gcd}(q,d_i){\rm gcd}(p,d_i)\mid d_i$ in \autoref{Main5} may be optimal for guaranteeing the conclusion.
\begin{theorem}\label{Main5}
Let $0<d_1\mid\cdots\mid d_r$ be the invariant factors of $B$. For any $p,q\in\mathbb{Z}_{>0}$, if  ${\rm gcd}(p,d_i){\rm gcd}(q,d_i)\mid d_i$ for all $1\le i\le r$ and $[\bm b]_{pq}\in\Col_{pq}(B)$, then
\[
\#M\big(\mathcal{A}_q^{(B,\bm b)}\big)\le \#M\big(\mathcal{A}_{pq}^{(B,\bm b)}\big).
\]
Moreover, when $l=0$, for all $p,q\in\mathbb{Z}_{>0}$, we have 
\[
\#M(\mathcal{A}_q)\le \#M(\mathcal{A}_{pq}).
\]
\end{theorem}
\begin{proof}
Notice that the map $\pi_q:\mathbb{Z}^n_{pq}\to\mathbb{Z}^n_q$ defined at the beginning of  \autoref{Sec3-2}, when restricted to $V_{pq}(B,\bm b)$, is a surjective map from  $V_{pq}(B,\bm b)$ to $V_q(B,\bm b)$ by \autoref{lem:surjective modular map}. With the simple fact that $\pi_q\big(H_{i,pq}^{(B,\bm b)}\big)\subseteq H_{i,q}^{(B,\bm b)}$ for each $i\in[m]$, we conclude that
\[
\pi^{-1}_q\Big(V_q(B,\bm b)-\bigcup_{i=1}^mH_{i,q}^{(B,\bm b)}\Big)\subseteq V_{pq}(B,\bm b)-\bigcup_{i=1}^mH_{i,pq}^{(B,\bm b)}.
\]
This immediately yields
\[
\#M\big(\mathcal{A}_q^{(B,\bm b)}\big)=\#\Big(V_q(B,\bm b)-\bigcup_{i=1}^mH_{i, q}^{(B,\bm b)}\Big)\le\#\Big(V_{pq}(B,\bm b)-\bigcup_{i=1}^mH_{i, pq}^{(B,\bm b)}\Big)=\#M\big(\mathcal{A}_{pq}^{(B,\bm b)}\big).
\]
Moreover, the relation $\#M(\mathcal{A}_q)\le \#M(\mathcal{A}_{pq})$ is trivial. We complete the proof.
\end{proof}
\section{Applications}\label{Sec4}
In this section, we primarily explore the applications of characteristic quasi-polynomials in graph coloring and flow problems. For this purpose, we begin by introducing some necessary concepts about graphs. The graphs $G$ with vertex set $V(G)$ and edge set $E(G)$ considered in this section are finite, allowing multiple edges and loops. For an edge subset $S\subseteq E(G)$, $G-S$ is the spanning subgraph obtained from $G$ by removing the edges in $S$. Denote by $G|S$ the spanning subgraph with edges $S$, that is, $G|S=G-(E(G)-S)$. The {\bf rank function} of $G$ is given by
\[
r(G|S):=|V(G)|-c(G|S),\quad \forall\, S\subseteq E(G),
\]
where $c(G)$ denotes the number of (connected) components of $G$. The {\bf nullity function} of  $G$ is defined by 
\[
m(G|S):=|S|-r(G|S)=|S|-|V(G)|+c(G|S),\quad\forall\,S\subseteq E(G).
\]
An {\bf orientation} of a graph is an assignment of direction to each edge. Every edge together with a direction is called an {\bf oriented edge}. A more general setting and background on graphs can be found in the book \cite{Bondy2008}. 

Graph coloring is a significant subfield in graph theory that originated in the middle of the 19th century with the famous Four Color Problem. A {\bf proper vertex coloring} of a graph $G$ is a vertex coloring $c:V(G)\to S$ such that $c(u)\ne c(v)$ whenever vertices $u$ and $v$ are adjacent. In 1912, Birkhoff \cite{Birkhoff1912} first introduced the chromatic polynomial to count the number of proper (vertex) colorings of planar graphs, which was extended to general graphs by Whitney \cite{Whitney1932, Whitney1932-1}.

Nowhere-zero $\mathbb{Z}_q$-flows or modular $q$-flows in graphs, were first defined by Tutte in \cite{Tutte1949,Tutte1954} as the dual concept to graph coloring. Correspondingly, Tutte also introduced  the flow polynomial function to count nowhere-zero flows. A comprehensive survey of nowhere-zero flows can be found in  \cite{Jaeger1988,Seymour1995}. Tutte further proposed several famous conjectures on nowhere-zero flows, known as the 5-flow conjecture \cite{Tutte1954}, 4-flow conjecture \cite{Tutte1966} and 3-flow conjecture (see Unsolved Problem 48, in \cite{Bondy1976}). These conjectures reveal that nowhere-zero flows are closely related to the edge connectivity of graphs. To address the 3-flow conjecture, Jaeger, Linial, Payan and Tarsi \cite{Jaeger1992} introduced the group connectivity of graphs in 1992 by generalizing the nowhere-zero flow to a nonhomogeneous form. Dually, they also defined group coloring as the dual concept of group connectivity, extending ordinary proper vertex colorings. Group connectivity and group colorings of graphs are nicely surveyed in \cite{LSZ2011}. 

From a different motivation, Zaslavsky \cite{Zaslavsky1995} defined and studied the balanced chromatic, dichromatic and Whitney number polynomials of biased graphs in 1995, which generalizes those of ordinary graphs. Specializing to gain graphs, he proved that the balanced chromatic polynomial of a gain graph precisely counts the proper colorations (also called group colorings). From a geometric perspective, ordinary proper colorings of a graph can be studied through its associated graphic arrangement. Extending this idea to gain graphs, Zaslavsky introduced affinographic arrangements to investigate their enumerative properties in \cite{Zaslavsky2003}, and showed that the balanced chromatic polynomial of a gain graph agrees with the characteristic polynomial of the associated affinographic arrangement. In \cite{FZ2007}, Forge and  Zaslavsky explored the number of integral proper colorations of a rooted integral gain graph. As one application, they provided a formula for calculating $\mathbb{Z}_q$-colorings even for small moduli, where Athanasiadi's general polynomial formula fails to apply. For further related topics, see \cite{FZ2016}. Inspired by their work, we realize that our method is adaptable to the enumerative study of $\mathbb{Z}_q$-colorings, and in \autoref{Modular-Coloring}, we obtain an alternative counting formula. The two treatments necessarily overlap, since they concern the same objects, but only partially; to keep the paper self-contained, we therefore give the full account we require.

Most recently, Kochol demonstrated in\cite{Kochol2022} that there exists a polynomial function counting nowhere-zero nonhomogeneous form flows, called the assigning polynomial. Subsequently, Fu, Ren and Wang \cite{FRW2025} provided an explicit formula for assigning polynomials and examined their coefficients. In fact, our approach is also well-suited to the study of nowhere-zero nonhomogeneous form flows.

From now on, we consider relevant topics associated to the additive cyclic group $\mathbb{Z}_q$ ($q\in\mathbb{Z}_{>0}$) and a fixed orientation $D$ of $G$. Let $w:E(G)\to\mathbb{Z}_q$. A {\bf $(\mathbb{Z}_q, w)$-coloring} of $G$ is a vertex coloring $c: V(G)\to\mathbb{Z}_q$ such that, for each oriented edge $uv$ (oriented from $u$ to $v$) of $G$, $c(v)-c(u)\ne w(uv)$. Naturally, when $w\equiv0$, $(\mathbb{Z}_q,0)$-coloring reduces to the ordinary proper $q$-coloring. Denoted by $C(G,w;\mathbb{Z}_q)$ the set of all $(\mathbb{Z}_q, w)$-colorings.

Let $b:V(G)\to\mathbb{Z}_q$. A {\bf $(\mathbb{Z}_q,b)$-flow} of $G$ is an edge-valued function $f:E(G)\to\mathbb{Z}_q$ such that
\[
\sum_{e \in E^{+}(v)} f(e)-\sum_{e \in E^{-}(v)} f(e)=b(v), \quad\forall\, v\in V(G).
\]
Here $E^+(v)$ denotes the set of edges for which $v$ is the head, and  $E^-(v)$ denotes the set of edges for which $v$ is the tail. Such a flow is {\bf nowhere-zero} if $f(e)\ne 0$ for all $e\in E(G)$. When $b\equiv0$, we recover the classical nowhere-zero $\mathbb{Z}_q$-flow. Let $F^*(G,b;\mathbb{Z}_q)$ denote the set of all nowhere-zero $(\mathbb{Z}_q,b)$-flows.

We concentrate on the enumerative study of group colorings and flows within the framework of truncated integral arrangements. To establish this connection, let us review the incidence matrix of a graph associated to its orientation. The {\bf incidence matrix} of $G$ is the $|V(G)|\times |E(G)|$ integral matrix $M_G:=(m_{ve})$ whose rows and columns are indexed by the vertices and edges, where, for a vertex $v$ and an edge $e$,
\[
m_{ve}:=\begin{cases}
1, & \text{ if } e \text{ is a link and } v \text{ is the head of } e;\\
-1,& \text{ if } e \text{ is a link and } v \text{ is the tail of } e;\\
0,& \text{ otherwise}.
\end{cases}
\]

For simplicity, we assume $V(G)=\{1,2,\ldots,n\}$ and $|E(G)|=m$. We denote the oriented edge from $i$ to $j$ as  $e_{ij}$ and its opposite orientation  $-e_{ij}$. We now consider the {\bf affinographic arrangement}
\[
\mathcal{A}_{(G,\bm w)}:=\big\{H_{e_{ij}}: x_j-x_i=w_{ij}\mid e_{ij}\in E(G)\big\},
\]
in $\mathbb{R}^n$ (see \cite{FZ2007,Zaslavsky2003}), whose defining matrix is $[M_G^T,\bm w]$ with $\bm w=(w_{ij})_{e_{ij}\in E(G)}\in\mathbb{Z}^m$. We need to be especially careful when $G$ contains a loop $e_{ii}\in E(G)$. If $w_{ii}=0$, such loop $e_{ii}$ corresponds precisely to the degenerate hyperplane $H_{e_{ii}}:x_i-x_i=0$, which is exactly the whole space $\mathbb{R}^n$. In this case, the characteristic polynomial $\chi(\mathcal{A}_{(G,\bm w)},t)$ is the zero polynomial. Naturally, $[\bm w]_q$ can be viewed as a function $w:E(G)\to\mathbb{Z}_q$ satisfying $w( e_{ij})=[w_{ij}]_q$ for all $e_{ij}\in E(G)$. In this context, the complement 
\[
M(\mathcal{A}_{(G,\bm w),q})=\mathbb{Z}_q^n-\bigcup_{e_{ij}\in E(G)}H_{e_{ij},q}
\]
of $\mathcal{A}_{(G,\bm w),q}$ is precisely the set of all $(\mathbb{Z}_q,w)$-colorings. Namely, the following relation holds:
\begin{equation}\label{Coloring-Complement}
C(G,w;\mathbb{Z}_q)=M(\mathcal{A}_{(G,\bm w),q}).
\end{equation}
Accordingly, the number of $(\mathbb{Z}_q,w)$-colorings of $G$ can be computed by enumerating the points in the complement of the $q$-reduction of the associated affinographic arrangement.

It is worth noting that $M_G$ is totally unimodular: every square submatrix has determinant in $\{-1,0,1\}$. Consequently, all invariant factors of $M_G$ are equal to $1$.  Hence, for any subset $S\subseteq E(G)$, the submatrix $[M_G]_S$ is also unimodular and all its invariant factors remain $1$. In addition, the following relation holds trivially:
\[
r\big([M_G^T]_S\big):={\rm rank}\big([M_G^T]_S\big)=r(G|S),\quad\forall\,S\subseteq E(G).
\]
According to the Interlacing Divisibility Theorem ( see \autoref{IDT}), for every subset $S\subseteq E(G)$, the first $r(G|S)$ invariant factors of $[M_G^T,\bm w]_S$ are equal to $1$. As with $q_0$, we set
\[
q_{\bm w}:=\max\Big\{d\big([M_G^T,\bm w]_S\big):S\subseteq E(G)\text{ and }r\big([M_G^T,\bm w]_S\big)\ne r(G|S)\Big\}, 
\]
where $d\big([M_G^T,\bm w]_S$ is the maximal invariant factor of $[M_G^T,\bm w]_S$. In analogy with $\tilde{d}_J(q)$, for any subset $S\subseteq E(G)$, we define 
\[
\tilde{d}_S(q):=
\begin{cases}
1,& \text{ if } r\big([M_G^T,\bm w]_S\big)= r(G|S);\\
1,&\text{ if } r\big([M_G^T,\bm w]_S\big)\ne r(G|S)\text{ and }d\big([M_G^T,\bm w]_S\big)\equiv0\pmod q;\\
0,& \text{ otherwise}.
\end{cases}
\]

With the above preparations, as a direct consequence of \autoref{Main1},  we can compute the $(\mathbb{Z}_q,w)$-colorings by the formula:
\[
\#C(G,w;\mathbb{Z}_q)=\sum_{S\subseteq E(G)}(-1)^{|S|}\tilde{d}_S(q)q^{c(G\mid S)}.
\]
Moreover, there exists a polynomial function $\chi(G,w;t)$ such that $\#C(G,w;\mathbb{Z}_q)=\chi(G,w;q)$ for all integers $q>q_{\bm w}$, and $\chi(G,w;t)$ can be written explicitly as
\[
\chi(G,w;t)=\sum_{S\subseteq E(G),\,r([M_G^T,\bm w]_S)=r(G\mid S)}(-1)^{|S|}t^{c(G\mid S)}.
\]
The polynomial  is called the {\bf modular chromatic polynomial} by Zaslavsky \cite{FZ2007}. According to \autoref{Char-Quasi}, $\chi(G,w;t)$ coincides with the characteristic polynomial of the affinographic arrangement $\mathcal{A}_{(G,\bm w)}$, which was first discovered in \cite{FZ2007}. Finally, \autoref{Main5} applied to \eqref{Coloring-Complement} yields the comparison relation: $\#C(G,w;\mathbb{Z}_q)\le \#C(G,w';\mathbb{Z}_{pq})$ for any $p,q\in\mathbb{Z}_{>0}$, where $w=[\bm w]_q$ and $w'=[\bm w]_{pq}$.

We summarize the results obtained so far as follows:
\begin{corollary}\label{Modular-Coloring}
Let $w:E(G)\to\mathbb{Z}_q$ with $q\in\mathbb{Z}_{>0}$ and $\bm w\in\mathbb{Z}^m$ satisfy $w=[\bm w]_q$. The following results hold:
\begin{itemize}
\item[{\rm(1)}] The counting formula of $C(G,w;\mathbb{Z}_q)$ is 
\[
\#C(G,w;\mathbb{Z}_q)=\sum_{S\subseteq E(G)}(-1)^{|S|}\tilde{d}_S(q)q^{c(G|S)}.
\]
Moreover, there exists a polynomial $\chi(G,w;t)$ such that $\#C(G,w;\mathbb{Z}_q)=\chi(G,w;q)$ for all integers $q>q_{\bm w}$, and $\chi(G,w;t)$ can be written as
\[
\chi(G,w;t)=\sum_{S\subseteq E(G),\,r([M_G^T,\bm w]_S)=r(G|S)}(-1)^{|S|}t^{c(G|S)}.
\]
\item[{\rm(2)}] $\chi(G,w;t)=\chi(\mathcal{A}_{(G,\bm w)},t)$. 
\item[{\rm(3)}]For any $p,q\in\mathbb{Z}_{>0}$, if $w'=[\bm w]_{pq}$, then $\#C(G,w;\mathbb{Z}_q)\le \#C(G,w';\mathbb{Z}_{pq})$.
\end{itemize}
\end{corollary}
Notably, part (1) in \autoref{Modular-Coloring} states that for all integers $q>q_{\bm w}$,  the number of $(\mathbb{Z}_q,w)$-colorings of $G$ is characterized by the polynomial $\chi(G,w;t)$.  Moreover, the formula for $\chi(G,w;t)$ shows that the modular chromatic polynomial is unchanged under any orientation $D'$ whenever it leads to the same family of edge subsets $S$ satisfying $r([M_G^T,\bm w]_S)=r(G|S)$ as the original orientation $D$. In fact, these edge subsets correspond exactly to the balanced subsets proposed by Zaslavsky. Thus, in Zaslavsky's terminology, the modular chromatic polynomials coincide under different orientations precisely when their balanced subsets are identical (see \cite{FZ2007}). Additionally, parts (2) and (3) in \autoref{Modular-Coloring} reveal that for all integers $q>q_{\bm w}$, we can compute $\chi(G,w;q)$ by $\chi(\mathcal{A}_{(G,\bm w)},t)$. This is the approach pioneered by Athanasiadis, in which it is not necessary to restrict $q$ to be a prime power as Athanasiadis did. The same phenomenon, expressed in the language of gain graphs, appears in \cite{FZ2007}.

Let $\bm b=(b_i)_{i\in V(G)}\in\mathbb{Z}^n$. Similarly, when a function $b:V(G)\to\mathbb{Z}_q$ satisfies $b(i)=[b_i]_q$ for all $i\in V(G)$, we understand it as the $q$-reduction of $\bm b$. In this view, the solution space $V_q(M_G,\bm b)$ of the system $M_G\bm x=\bm b\pmod q$  is exactly the set of all $(\mathbb{Z}_q,b)$-flows of $G$. Associated with $E(G)$, let $\mathcal{A}_{E(G)}$ be the {\bf coordinate arrangement} in $\mathbb{R}^m$ given by
\[
\mathcal{A}_{E(G)}:=\big\{H_e:x_e=0\mid e\in E(G)\big\},
\]
whose defining matrix is the matrix $[I,\bm 0]$, where $I$ denotes the $m\times m$ identity matrix whose rows and columns are indexed by the same edge labels as $M_G$. Recall that a $(\mathbb{Z}_q,b)$-flow $f$ is nowhere-zero if and only if $f(e)\ne 0$ for all edges $e$. Naturally, we consider the truncated arrangement $\mathcal{A}_{E(G)}^{(M_G,\bm b)}$ of $\mathcal{A}_{E(G)}$ by the matrix $[M_G,\bm b]$. Then the complement
\[
M\big(\mathcal{A}_{E(G),q}^{(M_G,\bm b)}\big)=V_q(M_G,\bm b)-\bigcup_{e\in E(G)}H_{e,q}^{(M_G,\bm b)}
\]
of the $q$-reduction of $\mathcal{A}_{E(G)}^{(M_G,\bm b)}$ corresponds precisely to the set of all nowhere-zero $(\mathbb{Z}_q,b)$-flows of $G$. Thus, the truncated arrangement $\mathcal{A}_{E(G)}^{(M_G,\bm b)}$ modulo any positive integer $q$  yields a framework that is adaptable to counting nowhere-zero $(\mathbb{Z}_q,b)$-flows. 

Let $M_S=\begin{bmatrix}M_G\\I_S\end{bmatrix}$ and $\bar{M}_S=\begin{bmatrix}M_G&\bm b\\I_S&\bm 0\end{bmatrix}$ for any $S\subseteq E(G)$. Because $M_G$ is unimodular, it is straightforward to check that every matrix $M_S$ is also unimodular and all its invariant factors are equal to $1$. Using the Interlacing Divisibility Theorem again, for every subset $S\subseteq E(G)$, 
the first $r(M_S)$ invariant factors of $\bar{M}_S$ are still equal to $1$. Notice the simple fact that
\[
 r(M_S):={\rm rank}(M_S)=m-m(G-S),\quad\forall\, S\subseteq E(G).
\]
Similar to the previous $q_{\bm w}$ and $\tilde{d}_S(q)$, we define
\[
q_{\bm b}:=\max\big\{d(\bar{M}_S):S\subseteq E(G)\text{ and }r(\bar{M}_S)\ne m-m(G-S)\big\}, 
\] 
and 
\[
\tilde{d}_S(q):=
\begin{cases}
1,& \text{ if } r(\bar{M}_S)=m-m(G-S);\\
1,&\text{ if }r(\bar{M}_S)\ne m-m(G-S)\text{ and }d(\bar{M}_S)\equiv0\pmod q;\\
0,& \text{ otherwise}.
\end{cases}
\]

With the above necessary groundwork, using the similar argument as \autoref{Modular-Coloring}, the next result follows immediately. 
\begin{corollary}\label{Modular-Flow}
Let $b:E(G)\to\mathbb{Z}_q$ with $q\in\mathbb{Z}_{>0}$ and $\bm b\in\mathbb{Z}^n$ satisfy $b=[\bm b]_q$. The following results hold:
\begin{itemize}
\item[{\rm(1)}] The counting formula of $F^*(G,b;\mathbb{Z}_q)$ is 
\[
\#F^*(G,b;\mathbb{Z}_q)=\sum_{S\subseteq E(G)}(-1)^{|S|}\tilde{d}_S(q)q^{m(G-S)}.
\]
Moreover, there exists a polynomial $\tau(G,b;t)$ such that $\#F^*(G,b;\mathbb{Z}_q)=\tau(G,b;q)$ for all integers $q>q_{\bm b}$, and $\tau(G,b;t)$ can be written as
\[
\tau(G,b;t)=\sum_{S\subseteq E(G),\,r(\bar{M}_S)=m-m(G-S)}(-1)^{|S|}t^{m(G-S)}.
\]
\item[{\rm(2)}] $\tau(G,b;t)=\chi\big(\mathcal{A}_{E(G)}^{(M_G,\bm b)},t\big)$. 
\item[{\rm(3)}]  For any $p,q\in\mathbb{Z}_{>0}$, if $b'=[\bm b]_{pq}$ and $G$ has a $(\mathbb{Z}_{pq},b')$-flow, then 
\[\#F^*(G,b;\mathbb{Z}_q)\le \#F^*(G,b';\mathbb{Z}_{pq}).\]
\end{itemize}
\end{corollary}

The counting function $\#F^*(G,b;\mathbb{Z}_q)$ and the polynomial $\tau(G,b;t)$ have analogous properties as $\#C(G,w;\mathbb{Z}_q)$ and $\chi(G,w;t)$, respectively;  we therefore omit their repetitive analysis and further exposition.
\section*{Acknowledgements}
The work is supported by National Natural Science Foundation of China (12301424).

\end{document}